\documentclass[final, reqno, 11pt]{amsart}
\usepackage{amsmath}
\usepackage{amsfonts}
\usepackage{amssymb}
\usepackage{amsthm}
\usepackage[colorlinks=true]{hyperref}

\newtheorem{theorem}{Theorem}[section]
\newtheorem{lemma}[theorem]{Lemma}
\newtheorem{proposition}[theorem]{Proposition}
\newtheorem{corollary}[theorem]{Corollary}
\theoremstyle{remark}
\newtheorem{observation}[theorem]{Remark}
\newtheorem{definition}[theorem]{Definition}

\newcommand{\les}{\lesssim}

\newcommand{\mc}{\mathcal}
\newcommand{\be}{\begin{equation}}
\newcommand{\ee}{\end{equation}}
\newcommand{\ba}{\begin{array}}

\newcommand{\ea}{\end{array}}
\newcommand{\bpm}{\begin{pmatrix}}
\newcommand{\epm}{\end{pmatrix}}
\newcommand{\lb}{\label}

\DeclareMathOperator{\Ran}{Ran}

\newcommand{\dd}{{\,}{d}}

\newcommand{\R}{\mathbb R}

\newcommand{\N}{\mathbb N}

\title[Positivity of Solutions]{A Positivity Criterion for the Wave Equation and Global Existence of Large Solutions}
\author{Marius Beceanu}
\address{University at Albany SUNY, Department of Mathematics and Statistics, Earth Science 110, Albany, NY, 12222, USA}
\email{mbeceanu@albany.edu}
\author{Avy Soffer}
\address{Rutgers University Department of Mathematics, 110 Frelinghuysen Rd., Piscataway, NJ, 08854, USA}
\email{soffer@math.rutgers.edu}
\begin{document}
\maketitle
\numberwithin{equation}{section}
\begin{abstract} In dimensions one to three, the fundamental solution to the free wave equation is positive. Therefore, there exists a simple positivity criterion for solutions. We use this to obtain large global solutions to two well-studied energy-supercritical semilinear wave equations, as well as some new results in the subcritical and critical cases.
\end{abstract}
\section{Introduction}
In this paper, we present large data global existence results for two energy-supercritical wave equations on $\R^{3+1}$. Both results are based on a simple positivity criterion for solutions to the free wave equation on $\R^{3+1}$.

Similar results hold in dimensions one and two, but for simplicity we focus on the three-dimensional case. For our convenience, some results are stated for smooth or classical solutions only (a priori), while others are stated for rough initial data.

\subsection{Quadratic nonlinearity (Nirenberg's equation)} The first equation we study is the quadratic semilinear wave equation on $\R^{3+1}$ satisfying the null condition
\be\lb{eq_fritz}
u_{tt} - \Delta u = u_t^2 - |\nabla u|^2,~u(0)=u_0,~u_t(0)=u_1.
\ee

Equation (\ref{eq_fritz}) is $\dot H^{3/2} \times \dot H^{1/2}$-critical, i.e.~$L^\infty$-critical. The nonlinearity is quadratic in $\nabla_{t, x} u$. Two is the Strauss exponent, meaning that there exist quadratic nonlinearities in $\nabla_{t, x} u$ (e.g.~$u_t^2$) that lead to finite time blow-up for arbitrarily small initial data, but all higher-order nonlinearities produce global solutions for sufficiently small initial data.

However, equation (\ref{eq_fritz}) has a better than expected behavior, because it satisfies the null condition (see \cite{klainer} or \cite{christo}; in fact, this equation is the canonical example of an equation that satisfies the null condition). Therefore, small initial data lead to global solutions.

For large data, global solutions to wave equations satisfying the null condition have been constructed by e.g.~\cite{wang}, \cite{shiwu}, \cite{mpy}, \cite{loy}. The paper \cite{li} proves global well-posedness for arbitrary large initial data for a supercritical wave equation, using its specific structure (as we also do below, but for another equation and using a different structure).

The most important fact about equation (\ref{eq_fritz}) is that by the substitution $v=e^{-u}$ it reduces to a linear wave equation on $\R^{3+1}$, that is $v_{tt}-\Delta v=0$. This transformation was suggested by Nirenberg and was first used in \cite{klainerman}, so we call (\ref{eq_fritz}) Nirenberg's equation.

Consequently, this equation has an infinite number of conserved quantities, i.e.~$\|(e^{-u}-1, e^{-u} u_t)\|_{\dot H^k \times \dot H^{k-1}}$. However, these quantities are not coercive, so cannot be used to always prove the global existence of solutions.

For example, the conserved energy has the form
$$
E[u](t)=\int_{\R^3 \times \{t\}} e^{-2u} (u_t^2 + |\nabla u|^2) \dd x.
$$
If $u \to +\infty$ then energy no longer controls the $\dot H^1 \times L^2$ norm of the solution. The same applies to the higher order conserved quantities.

Conserved quantities do not preclude the finite time blow-up of solutions to (\ref{eq_fritz}). In fact, there can be ODE-type blow-up\footnote{This type of blow-up can be achieved by considering constant initial data, which here lead to the ordinary differential equation $u_{tt}=u_t^2$. Solutions to this ODE always blow up in finite time for $u_t(0)=u_1>0$. Exploiting the finite propagation speed of solutions to the original equation (\ref{eq_fritz}), one can introduce a space cutoff and thus obtain finite energy solutions that blow up on an open set. See \cite{kel}.} for sufficiently large initial data.

Although relatively simple to study, due to the presence of infinitely many conserved quantities, equations (\ref{eq_fritz}) and (\ref{eq_fritz_general}) can serve as a model for more complicated semilinear equations, telling us what to expect in those cases.

In this paper we state sufficient (and in some cases necessary) conditions for the existence of global solutions to (\ref{eq_fritz}). These are classical solutions (i.e.~they satisfy the equation pointwise). For convenience, here results are stated a priori, assuming infinite regularity. Sharp versions will be stated elsewhere.

We need not assume finite energy (though energy is always locally finite), but there is also a stronger result about finite energy solutions, Proposition \ref{prop_dispersion}.

The first result refers to radially symmetric solutions.
\begin{proposition}[Main result]
\lb{fritz} Consider smooth radial initial data $(u_0, u_1)$ such that
\be\lb{cond}
(u_0)_r + |u_1| < \frac 1 r.
\ee
Then the equation (\ref{eq_fritz})
$$
u_{tt}-\Delta u=u_t^2-(\nabla u)^2,~u(0)=u_0,~u_t(0)=u_1
$$
admits a global smooth solution $u$ on $\R^{3+1}$ such that for every $(r, t) \in \R^{3+1}$
$$
u_r(r, t) + |u_t(r, t)| < \frac 1 r.
$$
Furthermore, if $u_0 \in L^\infty$, then $r (u_0)_r,~r u_1 \in L^\infty$ and $(u_0)_r + |u_1| \leq \frac {1-\epsilon} r$ for some $\epsilon>0$ if and only if $u \in L^\infty_{t, x}$. In this case,
\be\lb{comparison}
\inf u_0 - \ln (1 + \|r (u_0)_r\|_{L^\infty} + \|r u_1\|_{L^\infty}) \leq u \leq \sup u_0 + \ln(1/\epsilon).
\ee
Also, let $c_0:=e^{\sup u_0 - \inf u_0} \epsilon^{-1} (1 + \|r (u_0)_r\|_{L^\infty} + \|r u_1\|_{L^\infty})$. Then $ru_r$, $ru_t \in L^\infty_{t, x}$ and
\be\lb{detaliat}
u_r(r, t) + |u_t(r, t)| \leq \frac 1 r \Big(1-\frac 1 {c_0}\Big),~|u_r(r, t)| + |u_t(r, t)| \leq \frac {c_0-1} r. 
\ee
\end{proposition}

Our conditions are true for all small initial data (in a suitable sense) and also allow for a wide class of large initial data, e.g.~by taking $(u_0)_r$ to be large and negative.

One can prove along the same lines that the solutions depend continuously on the initial data. Moreover, under reasonable assumptions the solutions have finite energy and disperse:
\begin{proposition}\lb{prop_dispersion}
Consider equation (\ref{eq_fritz}) with smooth radial initial data $(u_0, u_1)$. If $u_0 \in L^\infty$, $(u_0, u_1) \in \dot H^1 \times L^2$, and $(u_0)_r + |u_1| \leq \frac {1-\epsilon} r$ for some $\epsilon>0$, then $u \in L^\infty_t \dot H^1_x \cap L^2_t L^\infty_x$, $u_t \in L^\infty_t L^2_x$, and
\be\lb{est}\begin{aligned}
\sup_t \|(u(t), u_t(t))\|_{\dot H^1 \times L^2} &\leq e^{\sup u_0 - \inf u_0} \epsilon^{-1} \|(u_0, u_1)\|_{\dot H^1 \times L^2},\\
\|u\|_{L^2_t L^\infty_x} &\les \max(1, e^{\sup u_0} \epsilon^{-1}) e^{-\inf u_0} \|(u_0, u_1)\|_{\dot H^1 \times L^2}.
\end{aligned}\ee
\end{proposition}
One can also obtain control of higher norms.

Condition (\ref{cond}) is optimal, meaning that it is necessary for global existence of solutions:
\begin{proposition}\lb{blowup} Consider smooth radial initial data $(u_0, u_1)$ such that for some $r_0 > 0$ $(u_0)_r(r_0) + |u_1(r_0)| \geq \frac 1 {r_0}$. Then the corresponding solution to equation (\ref{eq_fritz}) on $\R^3 \times I$ blows up in finite time, at a time $t_0$ with $|t_0| \leq r_0$.

More precisely, there exist $t_0$ with $|t_0| \leq r_0$ and $x_0 \in \R^3$ such that for $t$ close to $t_0$
$$\|u(t)\|_{L^\infty_x(|x - x_0| \leq 1)} \geq C + |\ln|t-t_0||,~\|u(t)\|_{H^{3/2}_x(|x - x_0| \leq 1)} \geq C |\ln|t-t_0||^{1/2}.
$$
\end{proposition}
We prove blow-up in the (critical for the equation) $L^\infty_{loc}$ and $H^{3/2}_{loc}$ senses, meaning that the solution becomes unbounded near a point. This also means it cannot be continued as a classical solution.

Nevertheless, the proof suggests that at least in some cases it may be possible to continue the solution past the blow-up point. This and more aspects of blow-up will be explored elsewhere.


The same ideas used in the study of equation (\ref{eq_fritz}) apply to the more general Nirenberg-type semilinear equation (also with a nonlinearity that satisfies the null condition)
\be\lb{eq_fritz_general}
u_{tt} - \Delta u = f(u) (u_t^2 - |\nabla u|^2),~u(0)=u_0,~u_t(0)=u_1.
\ee
For simplicity we assume that $f(u)$ is a smooth function, e.g.~$f(u)=1$, $f(u)=-u$, $f(u)=\sin u$, or $f(u)=-\arctan u$.

We introduce the auxiliary function
\be\lb{F}
F(t) = \int_0^t e^{-\int_0^s f(\sigma) \dd \sigma} \dd s.
\ee
$F$ is defined such that $F''/F'=-f$. If $u$ solves (\ref{eq_fritz_general}), then $F(u)$ solves the free wave equation.

Note that $F$ is strictly increasing, $F'>0$, and therefore injective. 

The subsequent discussion has to take into account whether $F(\pm \infty)$ is finite or infinite. For example, if $f(u)=1$ then $F(-\infty)=-\infty$, but $F(+\infty)$ is finite; if $f(u)=-1$ then the situation is reversed; if $f(u)=u$ then $F(\pm \infty)$ are finite; if $f(u)=-u$ or $f(u)=-\arctan u$ then $F(\pm \infty)=\pm\infty$. These four are all the possible cases.

In the radial case we have the following necessary and sufficient result:
\begin{proposition}\lb{prop_fritz_general} Consider smooth radial $(u_0, u_1)$. If $F(\pm \infty)=\pm \infty$, then equation (\ref{eq_fritz_general}) admits a corresponding smooth solution $u$ on $\R^{3+1}$. If in addition $u_0 \in L^\infty$, then $r(u_0)_r,~ru_1 \in L^\infty$ if and only if $u \in L^\infty_{t, x}$. In this situation, $r u_r,~r u_t \in L^\infty_{t, x}$ as well. 

If $F(-\infty) =a \in \R$, suppose that
\be\lb{cond1}
-(u_0)_r + |u_1| < \frac {F \circ u_0 - a}{r F' \circ u_0} = \frac {\int_{-\infty}^{u_0} e^{\int_s^{u_0} f(\sigma) \dd \sigma} \dd s} r.
\ee
If $F(+\infty) = b \in \R$, suppose that
\be\lb{cond2}
(u_0)_r + |u_1| < \frac {b-F \circ u_0}{rF'\circ u_0} = \frac {\int_{u_0}^\infty e^{\int_s^{u_0} f(\sigma) \dd \sigma} \dd s} r.
\ee
Then (\ref{eq_fritz_general}) admits a corresponding smooth solution $u$ on $\R^{3+1}$ such that
$$
-u_r + |u_t| < \frac {\int_{-\infty}^u e^{\int_s^u f(\sigma) \dd \sigma} \dd s} r \text{ and/or } u_r + |u_t| < \frac {\int_u^\infty e^{\int_s^u f(\sigma) \dd \sigma} \dd s} r.
$$

If $u_0 \in L^\infty$, then \{$r(u_0)_r,~ru_1 \in L^\infty$ and
\be\lb{Linfty}
-(u_0)_r + |u_1| \leq \frac {\int_{-\infty}^{u_0} e^{\int_s^{u_0} f(\sigma) \dd \sigma} \dd s - \epsilon} r \text{ and/or } (u_0)_r + |u_1| \leq \frac {\int_{u_0}^\infty e^{\int_s^{u_0} f(\sigma) \dd \sigma} \dd s - \epsilon} r
\ee
for some $\epsilon>0$\} if and only if $u \in L^\infty_{t, x}$. In this case, one also has that $r u_r,~r u_t \in L^\infty_{x, t}$ and a condition similar to (\ref{Linfty}) holds for all $t$:
\be\lb{Linfty'}
-u_r + |u_t| \leq \frac {\int_{-\infty}^u e^{\int_s^u f(\sigma) \dd \sigma} \dd s - \tilde \epsilon} r \text{ and/or } u_r + |u_t| \leq \frac {\int_u^\infty e^{\int_s^u f(\sigma) \dd \sigma} \dd s - \tilde \epsilon} r.
\ee
\end{proposition}
This is a large data result that generalizes Proposition \ref{fritz}. All the conclusions can be made quantitative. Under similar conditions we can also obtain finite energy and dispersive solutions and control of higher norms.

Conditions (\ref{cond1}) and (\ref{cond2}) are optimal, since their failure leads to finite time blow-up, as in Proposition \ref{blowup}. We omit the very similar proof.


Let $D^2 f$ denote the Hessian matrix of $f$ and in general $D^n f$ be the tensor consisting of the $n$-th order derivatives of $f$. In the nonradial case, we obtain:
\begin{proposition}\lb{prop_nonradial} Consider smooth initial data $(u_0, u_1)$. If $F(\pm \infty)=\pm \infty$, then equation (\ref{eq_fritz_general})
$$
u_{tt} - \Delta u = f(u) (u_t^2 - |\nabla u|^2),~u(0)=u_0,~u_t(0)=u_1
$$
admits a corresponding smooth solution on $\R^{3+1}$. If in addition $D^2 u_0, \nabla u_1 \in L^{3/2, 1}$, then $u \in L^\infty_{t, x}$.

If $F(-\infty) = a \in \R$, but $F(+\infty)=+\infty$, suppose that $\inf u_0 > -\infty$ and $u_1 \geq |\nabla u_0|$. Then equation (\ref{eq_fritz_general}) admits a corresponding smooth solution $u$ on $\R^3 \times [0, \infty)$ with $u \geq \inf u_0$.

Alternatively, suppose that $(u_0, u_1)$ decay at infinity together with their derivatives and
\be\lb{cond_nonradial1}
-\Delta u_0 + f(u_0) (\nabla u_0)^2 \geq |\nabla u_1 - f(u_0) u_1 \nabla u_0|. 
\ee
Then equation (\ref{eq_fritz_general}) admits a corresponding smooth solution $u$ on $\R^{3+1}$ with $u \geq 0$. If in addition $D^2 u_0, \nabla u_1 \in L^{3/2, 1}$, then $u \in L^\infty_{t, x}$.

Similar results apply to the case when $F(-\infty)=-\infty$, but $F(+\infty) = b \in~\R$.
\end{proposition}

In particular, equation (\ref{eq_fritz}) also falls under the hypotheses of Proposition~\ref{prop_nonradial}.

Here $L^{3/2, 1}$ is a Lorentz space; for their definition and properties see \cite{bergh}. In terms of the more familiar Lebesgue spaces, one has that $L^{3/2-\epsilon} \cap L^{3/2+\epsilon} \subset L^{3/2, 1} \subset L^{3/2}$.

One can easily show (using the substitution $v=F(u)$) that the condition (\ref{cond_nonradial1}) allows for large initial data. This is obvious for the other condition.

Under similar conditions we can also obtain finite energy and dispersive solutions and control of higher norms.

In the nonradial case there is no expectation that our conditions are optimal. Nevertheless, we can obtain a more general result. A solution to (\ref{eq_fritz_general}) can be continued as long as $F(u)>a$ and/or $F(u)<b$ (and indefinitely if $F(\pm \infty) = \pm \infty$). As soon as $F(u)=a$ or $F(u)=b$, one has blow-up in the $L^\infty_{loc}$ and $H^{3/2}_{loc}$ sense, as in Proposition \ref{blowup}.

We next prove that all sufficiently nice solutions to equation (\ref{eq_fritz_general}) either disperse or blow up in finite time; solitons, infinite time blow-up, or other more complicated types of behavior are excluded. This can be construed as a version of the soliton resolution conjecture for this problem, except there are no solitons. More generally, such a result is called asymptotic completeness.

\begin{proposition}[Asymptotic completeness]\lb{prop_general} For Schwartz-class initial data $(u_0, u_1)$, equation (\ref{eq_fritz_general})
$$
u_{tt} - \Delta u = f(u) (u_t^2 - |\nabla u|^2),~u(0)=u_0,~u_t(0)=u_1
$$
always admits a smooth solution $u$ that either blows up in finite time (in $L^\infty_{loc}$ and $H^{3/2}_{loc}$, see Proposition \ref{blowup}) or is globally defined on $\R^{3+1}$. In the latter case $u(t)$ is a Schwartz-class function for each $t \in \R$, $\sup_t \|u(t)\|_{H^n} < \infty$ for each $n$, $u$ and all its derivatives disperse: $\|D^n u(t)\|_{L^\infty} \les_n |t|^{-1}$.\\
Finally, $u$ scatters, meaning it behaves like a solution of the free wave equation as $t \to \pm \infty$: there exist Schwartz-class $v_0$, $v_1$ such that, if $v$ is the corresponding solution of the free wave equation on $\R^{3+1}$, for any $n \geq 0$
$$
\lim_{t \to +\infty} \|u(t)-v(t)\|_{H^n} = 0.
$$
\end{proposition}

\begin{observation} For equation (\ref{eq_fritz}), if the solution $u$ blows up in finite time, then it must blow up at a time $T$ with
$$
|T| \leq \frac 1 {4\pi} e^{\|u_0\|_{L^\infty}} (\|\Delta u_0\|_{L^1} + \|\nabla u_0\|_{L^2}^2 + \|\nabla u_1\|_{L^{1}});
$$
otherwise $u$ exists globally and disperses. More generally, for equation (\ref{eq_fritz_general}), the blow-up time $T$ must satisfy
$$
|T| \leq \frac 1 {4\pi} \frac {\max(F(\max u_0), -F(\min u_0))}{\min(F(+\infty), -F(-\infty))} (\|\Delta u_0\|_{L^1} + \|\nabla u_0\|_{L^2}^2 + \|\nabla u_1\|_{L^{1}}),
$$
where $F$ is defined by (\ref{F}).
\end{observation}

In the radial case we have a more precise classification, with necessary and sufficient conditions for global existence (see Proposition \ref{prop_fritz_general}).

The assumption that $(u_0, u_1)$ are of Schwartz class is not needed; we only need control of finitely many seminorms of the initial data. The question of optimal norms is an interesting one, but will be examined elsewhere.

\begin{observation} The previous existence results remain valid, with the same proof, if we add to the equation (\ref{eq_fritz}) a smooth source term $G \leq 0$, i.e.
$$
u_{tt} - \Delta u = u_t^2-|\nabla u|^2 + G.
$$
Indeed, following the Nirenberg transformation $v=e^{-u}$, $G$ becomes a potential with a positive contribution to the solution:
$$
v_{tt} - \Delta v = -G v.
$$
Since our results in this case depended on $v$ being positive, adding another positive source term does not affect them.

For the more general Nirenberg-type equation (\ref{eq_fritz_general}), one can always add a magnetic potential of the form $A_1 u_t + A_2 \cdot \nabla u$, where $A_1$, $A_2$ are constants.

One can also add other terms to the equation, usually at the price of having to deal with a nonlinear equation even after the Nirenberg transformation. In general,
$$
u_{tt} - \Delta u = u_t^2-|\nabla u|^2 + F(|u|)u
$$
becomes
$$
v_{tt} - \Delta v = v F(|\ln v|) \ln v.
$$
In particular, though, the nonlinear term $c e^u$ transforms into $-c$, a source term that does not depend on $v$. Such extensions of our results will be dealt with in another paper.
\end{observation}

\begin{observation} A similar finite time blow-up/dispersion dichotomy is true for the energy-supercritical Schr\"{o}dinger equation
$$
iu_t-\Delta u=(\nabla u)^2,~u(0)=u_0.
$$
This is also a $\dot H^{3/2}$-critical equation and one can also use a Nirenberg transformation, $e^u=v$, to solve it. Note that the evolution is not unitary.
\end{observation}

Finally, although the main topic of this paper is the existence of global solutions, we also briefly look into the question of local existence. We state a result that holds in the general case; undoubtedly, there can be improvements for radial solutions.
\begin{proposition}\lb{prop_local_existence} Consider equation (\ref{eq_fritz_general}) with smooth initial data $(u_0, u_1)$ and assume that $F(-\infty)=a \in \R$, but $F(+\infty)=+\infty$. If $u_0,~\nabla u_0,~u_1 \in L^\infty$, then there exists a corresponding smooth solution $u$ on $\R^3 \times (-T, T)$, where
\be\lb{time}
T = \frac {F(\inf u_0) - a}{(\|\nabla u_0\|_{L^\infty} + \|u_t\|_{L^\infty})\sup F'(u_0)}\bigg.
\ee
Furthermore, $u \in L^\infty_{t, x}(\R^3 \times [-t, t])$ for any $t<T$.

Similar results hold in the case when $F(+\infty)=b \in \R$.
\end{proposition}
Conversely, following the argument in the proof, one can easily construct examples of smooth initial data $(u_0, u_1)$ with $u_0$, $\nabla u_0$, or $u_1 \not \in L^\infty$, which lead to blow-up at time $0$ (hence the failure of local existence). We omit the construction. For more such results concerning the lack of local existence, see \cite{lindblad2}.

\subsection{Monomial nonlinearity} The other equation we study in this paper is the focusing semilinear wave equation on $\R^{3+1}$ with a monomial nonlinearity
\be\lb{eq_sup}
u_{tt} - \Delta u - |u|^N u =0,~u(0)=u_0,~u_t(0)=u_1.
\ee
This equation is $\dot H^{1/2}$-critical for $N=2$, energy-critical for $N=4$, and energy-supercritical for $N>4$. In general, the equation is $\dot H^{s_c}$-critical, where $s_c=3/2-2/N$.

An equivalent formulation is
\be\lb{weak}
u(t) = \cos(t\sqrt{-\Delta}) u_0 + \frac {\sin(t\sqrt{-\Delta})}{\sqrt{-\Delta}} u_1 + \int_0^t \frac {\sin((t-s)\sqrt{-\Delta})}{\sqrt{-\Delta}} |u(s)|^N u(s) \dd s.
\ee
Functions $u$ that satisfy this identity are sometimes called mild solutions of (\ref{eq_sup}); if, in addition, they are sufficiently smooth, then they become classical solutions.

For equation (\ref{eq_sup}) we prove the existence of global strong solutions in the sense of (\ref{weak}) for a suitable class of large initial data. All solutions are Borel measurable functions and defined everywhere, instead of almost everywhere.

In some cases the range of the solutions is $\R$, but in other cases it is $[0, \infty]$, i.e.~they can be infinite on a set of positive measure, so they may be said to blow up in finite time in the usual sense. This is similar to (and due to the limitations of) the Lebesgue integral, which is defined either for finite-valued or for nonnegative functions --- and in the former case the function has to be dominated by an integrable nonnegative function.

In both cases, these will be solutions in the weak (distributional) sense, but in the nonnegative case only when tested against nonnegative test functions (in the sense that both sides will either be infinite or have the same finite value).

Baire measurable functions are, by definition, those functions that are pointwise limits of sequences of continuous functions. On Euclidean space, a function is Baire measurable if and only if it is Borel measurable.
\begin{lemma} Let $f:\R^3 \times [0, \infty) \to [0, \infty]$ be a Borel measurable function. Then
$$
F_f(y, t) := \int_0^t \frac {\sin((t-s)\sqrt{-\Delta})}{\sqrt{-\Delta}} f(x, s) \dd s
$$
is also nonnegative and Borel measurable. If in addition $f$ is lower semicontinuous, then $F_f$ is also lower semicontinuous.
\end{lemma}
\begin{proof} As the kernel of the sine propagator is a positive measure, the proof is based on monotone convergence. First, it is enough to show $F_{f_n}$ is Borel measurable for $f_n=\max(f, n)$ for each $n \in \N$, since $f$ is the increasing pointwise limit of $f_n$.

Next, for fixed $n$, $f_n$ is the pointwise limit of some sequence of uniformly bounded continuous functions $0 \leq g_m \leq n$. But then, by dominated convergence, $F_{g_m}(y, t) \to F_{f_n}(y, t)$ for each $y$ and $t$. Since $F_{g_m}$ are Borel measurable (continuous, in fact), so is $F_{f_n}$.

If $f \geq 0$ is lower semicontinuous, then it is the increasing limit of a sequence of continuous functions $f_n \geq 0$. Thus, $F_f$ will be the increasing pointwise limit of $F_{f_n}$, which are continuous, so $F_f$ will be lower semicontinuous too.
\end{proof}

As an application, we state a criterion for the global existence, dispersion, and scattering of solutions to equation (\ref{eq_sup}) in the energy-critical case $N=4$. This criterion extends the well-known one of \cite{kenig}, see below.

We also obtain a different criterion for global existence, dispersion, and scattering that applies to the whole $\dot H^{1/2}$-supercritical ($N>2$) range.

The most important feature of equation (\ref{eq_sup}) that we use is that the nonlinearity is increasing. Conceivably, the same methods could also work for other types of nonlinearities.

Some of our results also hold in the defocusing case, i.e.~when the nonlinearity in (\ref{eq_sup}) has the plus sign.

Equation (\ref{eq_sup}) is well-studied. We only give a brief survey of the known results.

The global existence of solutions for small initial data was proved by, among others, \cite{pecher}, \cite{lindblad}, and \cite{shst}. This equation can have ODE-type blow-up in finite time for any $N>0$, for sufficiently large initial data. Another approach to blow-up is due to \cite{levine}.

For $N=4$, the energy-critical case, equation (\ref{eq_sup}) also has soliton solutions $u(x, t)=W(x) \in \dot H^1$, where $W$ solves the semilinear elliptic equation
\be\lb{w}
-\Delta W=W^5.
\ee
We distinguish the ground state soliton $Q>0$ given by the explicit formula
\be\lb{Q}
Q(x) = \frac 1 {\bigg(1+\displaystyle\frac {|x|^2} 3\bigg)^{1/2}}.
\ee
In fact, each such soliton is part of an infinite family of solitons obtained by rescaling, but this will play no role in our proof.

Recently, \cite{kenig} classified all solutions smaller than the ground state soliton for the energy-critical equation. Define the energy of a solution by
\be\lb{energy_focusing}
E[u]=\int_{\R^3 \times \{t\}} \frac {u_t^2 + |\nabla u|^2} 2 - \frac {u^6} 6 \dd x.
\ee
The result of \cite{kenig} states that if $E[u] < E[Q]$ and $\|\nabla u_0\|_{L^2} < \|\nabla Q\|_{L^2}$ then the solution scatters, while if $E[u] < E[Q]$ and $\|\nabla u_0\|_{L^2} > \|\nabla Q\|_{L^2}$ then the solution blows up.

This classification was then extended, in a less precise manner, by \cite{dkm} to radial solutions of arbitrary size. These results also extend to the energy-supercritical case (where the lack of solitons makes the classification simpler), but only under the assumption that the solution stays bounded in the critical Sobolev norm; see \cite{dkm2}, \cite{dola}, \cite{duro}.

Another approach to the energy-critical equation (\ref{eq_sup}) belongs to \cite{krsc}, \cite{kst}, \cite{kns}, \cite{kns3}, \cite{bec}, and \cite{kns2}, which studied solutions in a neighborhood of the ground state soliton.

Results for the focusing supercritical equation (\ref{eq_sup}) include \cite{krsc2}, \cite{becsof}, and \cite{loy}. These papers construct particular classes of large global solutions.

Many more results are known in the defocusing case
$$
u_{tt} - \Delta u + |u|^N u =0,~u(0)=u_0,~u_t(0)=u_1
$$
((\ref{eq_sup}) with a plus in front of the nonlinearity). For the energy-critical equation,  global well-posedness was proved by \cite{struwe} and \cite{grillakis}. In the energy-supercritical case, some results --- \cite{tao}, related to an idea from \cite{gsv}, \cite{roy1}, \cite{roy2}, and \cite{struwe2} --- refer to slightly supercritical equations, while others --- \cite{keme2}, \cite{kivi1}, \cite{kivi2}, \cite{miao}, \cite{bul1}, \cite{bul2}, \cite{bul3} --- are conditional results. The paper \cite{becsof2} uses a modified supercritical nonlinearity. Also see the recent result \cite{tao2}, which shows blow-up for a defocusing system of supercritical wave equations.

We state our next couple of results for a class of rough initial data and rough solutions. Henceforth, we assume that initial data are Borel measurable and our hypotheses about them hold a.e., as spelled out below in Proposition \ref{maximum}.

We also adopt the following convention: a backward light cone is one whose base is at $t=0$, while a forward light cone is one opening away from $t=0$. Forward and backward light cones intersect along sets of zero measure on each.

The following is the statement of our most general result in this context, which is based on monotone convergence.

\begin{proposition}\lb{maximum} Assume that $N \geq 0$, $(u_0, u_1)$ are a pair of initial data, and \\
i. either radial and outgoing according to Definition \ref{def_outgoing}, $u_0$ is Borel measurable, and $u_0 \geq 0$ a.e.;\\
ii. or radial, $u_0$, $(u_0)_r$, and $u_1$ are Borel measurable, and $(u_0)_r + u_0/r \geq |u_1|$ a.e.;\\
iii. or not necessarily radial, $u_0$, $\nabla u_0$, and $u_1$ are Borel measurable, and $u_0 \geq 0$, $u_1 \geq |\nabla u_0|$ a.e.;\\
iv. or not necessarily radial, $-\Delta u_0$ and $\nabla u_1$ are Borel measurable, and $-\Delta u_0 \geq |\nabla u_1|$.

Then there exists a global solution $u$, on $\R^3 \times [0, \infty)$ in cases i and iii and on $\R^{3+1}$ in cases ii and iv, to equation (\ref{weak}), having $(u_0, u_1)$ as initial data. The solution $u$ is Borel measurable and $\Ran u \subset [0, +\infty]$. Moreover, $u$ is the smallest nonnegative solution to equation (\ref{weak}) with initial data $(u_0, u_1)$.

Consider two solutions $u$ and $v$ obtained by this method. In the first case if $u_0 \geq v_0 \geq 0$ a.e., in the second case if
$$
(u_0)_r + u_0/r \pm u_1 \geq (v_0)_r + v_0/r \pm v_1 \geq 0~a.e.,
$$
in the third case if
$$
u_0 \geq v_0,~u_1-|\nabla u_0| \geq v_1 + |\nabla v_0| \geq v_1 - |\nabla v_0| \geq 0~a.e.,
$$
and in the fourth case if
$$
-\Delta u_0 - |\nabla u_1| \geq -\Delta v_0 + |\nabla v_1| \geq -\Delta v_0 - |\nabla v_1| \geq 0~a.e.,
$$
then one has for the corresponding solutions that $u \geq v \geq 0$.

Let $u_{lin}$ be the linear evolution of the initial data. If $u_{lin}$ is lower semicontinuous, then $u$ is lower semicontinuous. If $(u_n)_{lin} \uparrow u_{lin}$ pointwise a.e.~on each backward light cone (see Lemma \ref{pointwise}), then $u_n \uparrow u$ pointwise a.e.~on each backward light cone.
\end{proposition}

\begin{proof}
Consider any other solution $\tilde u \geq 0$ to (\ref{weak}) with the same initial data $(u_0, u_1)$. Clearly then $\tilde u \geq u_{lin}=u^1$ and by induction we get that $\tilde u \geq u^n$ for each $n$, so $\tilde u \geq \lim_{n \to \infty} u^n = u$.

Consider a sequence $(u_n)_{lin} \uparrow u_{lin}$ pointwise a.e.~on each backward light cone. Then by induction $(u_n)^m \uparrow u^m$ for each $m$, so in the limit $u_n \uparrow u$ pointwise a.e.~on each backward light cone.
\end{proof}

Thus, at least in some cases, solutions depend monotonically on the initial data. Since we construct them explicitly, these solutions to (\ref{weak}) are uniquely defined for each pair of initial data. They also have a unique characterization, as the smallest nonnegative solutions to (\ref{weak}). However, they are not unique in this class: to give a trivial example, $u(x, t)=+\infty$ on $\R^3 \times \R^*$ is a solution to (\ref{weak}) for any $(u_0, u_1)$. 

In this generality, solutions produced by Proposition \ref{maximum} are weak solutions and some of them actually blow up in finite time in the usual sense. However, as we shall prove below, imposing stronger conditions on such solutions leads to better behavior, including regularity and uniqueness.

Our approach to the study of equation (\ref{eq_sup}) seems to be new. One can hope that this unified treatment of global existence and blow-up will lead to a better understanding of the boundary between them.

After monotone convergence, the next step is using dominated convergence. This works both in the focusing and in the defocusing case, i.e.~for
\be\lb{classical_pm}
v_{tt} - \Delta v \pm |v|^N v =0,~v(0)=v_0,~u_t(0)=v_1.
\ee
In fact, the $\pm$ sign can be replaced by any number on the unit circle.

For the next result we still focus on mild solutions, which satisfy the equivalent (for classical solutions) formulation
\be\lb{weakpm}
v(t) = \cos(t\sqrt{-\Delta}) v_0 + \frac {\sin(t\sqrt{-\Delta})}{\sqrt{-\Delta}} v_1 \pm \int_0^t \frac {\sin((t-s)\sqrt{-\Delta})}{\sqrt{-\Delta}} |v(s)|^N v(s) \dd s.
\ee
Reverting to the usual notions of solution existence and blow-up, one has the following criterion (whose proof we omit):
\begin{proposition}\lb{global_existence} Assume $N \geq 0$ and consider initial data $(u_0, u_1)$ and $(v_0, v_1)$, such that $(u_0, u_1)$ give rise to a nonnegative, Borel measurable, a.e. finite solution $u$ for equation (\ref{weak}) on $\R^3 \times I$, on some interval $I$ around $t=0$. If almost everywhere\\
i. $(u_0, u_1)$ and $(v_0, v_1)$ are radial and outgoing and $u_0 \geq |v_0|$\\
ii. or $(u_0, u_1)$ and $(v_0, v_1)$ are radial and $(u_0)_r + u_0/r - |u_1| \geq |(v_0)_r + v_0/r| + |v_1|$\\
iii. or $u_1-|\nabla u_0| \geq |v_1| + |\nabla v_0|$\\
iv. or $-\Delta u_0-|\nabla u_1| \geq |\Delta v_0| + |\nabla v_1|$,\\
then $(v_0, v_1)$ also give rise to an a.e.~finite solution $v$ on $\R^3 \times I$ to the equation (\ref{weakpm}) with either the focusing or the defocusing sign, such that $|v| \leq u$.

Let $v_{lin}$ be the linear evolution of initial data. Then, when $(v_n)_{lin} \to v_{lin}$ pointwise a.e.~on each backward light cone, see Lemma \ref{pointwise}, it follows that $v_n \to v$ pointwise a.e.~on each backward light cone.

Finally, assume in addition that $(v_0, v_1)$ are real-valued. Then the solution $v$ of the focusing equation is unique in the class $|v| \leq u$.
\end{proposition}

If $u$ is a.e.~finite, it follows that $u$ is also integrable, hence a.e.~finite, on each backward light cone in $\R^3 \times I$. By our construction, then, the dominated solution $v$ will have well-defined pointwise values, possibly except at those points.

Here we assume that $I \subset [0, \infty)$ in cases i and iii.

If $u$ disperses etc.~then so does $v$. If a Lebesgue or Strichartz norm of $v$ blows up in finite time, then the same is true for $u$.

In order to take full advantage of this criterion, we need to know beforehand that a certain positive solution $u$ exists globally or that a solution $v$ blows up. The simplest choice is the ground state soliton $Q$ defined by (\ref{Q}):
$$
Q(x) = \frac 1 {\bigg(1+\displaystyle\frac {|x|^2} 3\bigg)^{1/2}}.
$$
Note that both $(r Q)_r > 0$ and $-\Delta Q = Q^5 > 0$. A more refined choice for comparison is $(1\pm \epsilon) Q$.

The following is a straightforward corollary of Proposition \ref{global_existence}.

\begin{corollary}\lb{critic} Assume $N=4$ and consider initial data $(u_0, u_1)$ such that\\
i. either (in the radial case) $|(u_0)_r+u_0/r|+|u_1| \leq Q_r+Q/r$\\
ii. or (in the general case) $|\Delta u_0| + |\nabla u_1| \leq -\Delta Q = Q^5$.\\
Then $(u_0, u_1)$ give rise to a global solution $u$ of equation (\ref{weakpm}) on $\R^{3+1}$
$$
v(t) = \cos(t\sqrt{-\Delta}) v_0 + \frac {\sin(t\sqrt{-\Delta})}{\sqrt{-\Delta}} v_1 \pm \int_0^t \frac {\sin((t-s)\sqrt{-\Delta})}{\sqrt{-\Delta}} |v(s)|^4 v(s) \dd s,
$$
such that $|u(x, t)| \leq Q(x)$.

If in addition $|(u_0)_r+u_0/r|+|u_1| \leq (1-\epsilon)(Q_r+Q/r)$ or $|\Delta u_0| + |\nabla u_1| \leq (1-\epsilon)Q^5$, then $u$ disperses ($\|u\|_{L^8_{t, x}} < \infty$) and scatters.

If on the contrary $(u_0)_r+u_0/r-|u_1| \geq (1+\epsilon)(Q_r+Q/r)$ or $-\Delta u_0 - |\nabla u_1| \geq (1+\epsilon)Q^5$, then the solution $u$ of (\ref{weak}) blows up in finite time, in the sense that there exist finite $a<0<b$ such that $\|u\|_{L^8_{t, x}(\R^3 \times (a, 0))} = \|u\|_{L^8_{t, x}(\R^3 \times (0, b))} = +\infty$.
\end{corollary}

\begin{observation} This criterion goes in some cases beyond the result of \cite{kenig}. Indeed, consider a solution $u$ with initial data $(u_0=0, u_1=Q_r+Q/r)$. Since the first component is zero, the energy of $u$, defined by (\ref{energy_focusing}), is given by
$$
E[u] = \frac 1 2 \int_{\R^3} u_1^2 \dd x = 2\pi \int_0^\infty \bigg(\frac 1 {(1+r^2/3)^{3/2}r}\bigg)^2 r^2 \dd r = \frac {3\sqrt 3 \pi^2} 8.
$$
At the same time,
$$
\frac 1 2 \|\nabla Q\|_{L^2}^2 = \frac 1 2 \int_{\R^3} Q_r^2 \dd x = 2\pi \int_0^\infty \bigg(\frac {r/3}{(1+r^2/3)^{3/2}}\bigg)^2 r^2 \dd r = \frac {3\sqrt 3 \pi^2} 8.
$$
One can also see directly that
$$
\int_{\R^3} (2Q_rQ/r+Q^2/r^2) \dd x = 4\pi \int_0^\infty 2Q_rQr+Q^2 \dd r = 0,
$$
so
$$
\|\nabla Q\|_{L^2}^2 = \int_{\R^3} Q_r^2 \dd x = \int_{\R^3} (Q_r+Q/r)^2 \dd x = \|u_1\|_{L^2}^2.
$$
Since $\|\nabla Q\|_{L^2}^2=\|Q\|_{L^6}^6$, it follows that $E[Q] = \frac 1 3 \|\nabla Q\|_{L^2}^2$, so $E[u] = \frac 3 2 E[Q]$.

The result of \cite{kenig} refers only to solutions $u$ for which $E[u] \leq E[Q]$, so it does not apply to $u$. Nor does the result of \cite{kns3} apply, since it refers to solutions $u$ for which $E[u] \leq E[Q] + \epsilon^2$, $\epsilon << 1$.

On the other hand, Corollary \ref{critic} shows that $u$ exists globally on $\R^{3+1}$. If we consider instead the solution $u_{\epsilon}$ with initial data $(u_0=0, u_1=(1- \epsilon)(Q_r+Q/r))$, then Corollary \ref{critic} tells us that $u_\epsilon$ exists globally and disperses.
\end{observation}

It is also possible to construct a small perturbation $u_0$ of $Q$ such that $0 \leq -\Delta u_0 \leq (1-\epsilon)(-\Delta Q)$ or $-\Delta u_0 \geq (1+\epsilon)(-\Delta Q)$ and $E[(u_0, 0)]>E[Q]$. In both cases, the behavior of the solution can be predicted using Corollary \ref{critic}, but not using the results of \cite{kenig}. The connection with \cite{kns3} remains to be studied.

Corollary \ref{critic} depends essentially on the ground state soliton $Q$, which is only guaranteed to exist in the energy-critical case $N=4$, and on the results of \cite{kenig}.

However, more generally, for $N>2$ there exists another family of positive stationary solutions to equation (\ref{eq_sup}), which can be called \emph{singular solitons}.

Namely, we look for positive radial solutions of the form $Q_N=C_N|x|^\alpha$ of the semilinear elliptic equation (\ref{w})
$$
-\Delta Q_N = Q_N^{N+1}.
$$
Plugging our ansatz in the equation, we obtain
$$
-\alpha(\alpha+1)C_N r^{\alpha-2} = C_N^{N+1} r^{\alpha(N+1)}.
$$
Hence $\alpha=-\frac 2 N$ and $C_N=\big(-\alpha(\alpha+1)\big)^{1/N}=\big(2(N-2)/N^2\big)^{1/N}$, so
\be\lb{qn}
Q_N(x) = \bigg(\frac {2(N-2)}{N^2}\bigg)^{1/N} |x|^{-2/N}.
\ee
Note that $C_N>0$ implies that $N>2$. $Q_N$ is scaling-invariant, so by rescaling it we do not obtain anything new, but we can get other solitons by translating it.

$Q_N$ never has finite energy, but it has locally finite energy for $N>4$. $Q_N$ logarithmically fails to be in the critical Sobolev space $\dot H^{s_c}$, where $s_c=3/2-2/N$, being instead in the Besov space $\dot B^{s_c}_{2, \infty} \subset L^{p_c, \infty}$.

$N>2$ is the $\dot H^{1/2}$-supercritical case ($s_c>1/2$), which includes the energy-critical and energy-supercritical cases. In particular, in the energy-critical case, $Q_4(x) = \frac 1 {\sqrt 2} |x|^{-1/2}$. Comparing this singular soliton with the finite energy soliton $Q$ defined by (\ref{Q}), we see that $Q_4(x) \leq Q(x)$ for $|x| \in [3-\sqrt 6, 3 + \sqrt 6]$ and $Q_4(x) > Q(x)$ otherwise. Of course, this depends on the scaling of $Q$, but regardless of scaling we see that neither soliton dominates the other.

Using the singular soliton $Q_N$ we can formulate another criterion for global existence of solutions to (\ref{weakpm}). Note that $(rQ_N)_r>0$ and $-\Delta Q_N=Q_N^{N+1} > 0$.
\begin{corollary}\lb{supercritic} Assume $N>2$ and consider initial data $(u_0, u_1)$ to equation (\ref{weakpm})
$$
u(t) = \cos(t\sqrt{-\Delta}) u_0 + \frac {\sin(t\sqrt{-\Delta})}{\sqrt{-\Delta}} u_1 \pm \int_0^t \frac {\sin((t-s)\sqrt{-\Delta})}{\sqrt{-\Delta}} |u(s)|^N u(s) \dd s.
$$
i. If $|(u_0)_r+u_0/r|+|u_1| \leq (Q_N)_r+Q_N/r$ (in the radial case)\\
ii. or if $|\Delta u_0| + |\nabla u_1| \leq -\Delta Q_N = Q_N^{N+1}$ (in the general case),\\
then the corresponding solution $u$ exists globally on $\R^{3+1}$ and $|u(x, t)| \leq Q_N(x)$. Moreover, in the focusing case, $u$ is the unique solution of (\ref{weak}) such that $|u| \leq Q_N$.

At each time $t$, $u(t) - u_{lin}(t) \in \dot B^{s_c-1}_{2, \infty}$, where $u_{lin}$ is the linear evolution of $(u_0, u_1)$, with
$$
\|u(t) - u_{lin}(t)\|_{\dot B^{s_c-1}_{2, \infty}} \les t.
$$
\end{corollary}
The initial data need not have finite energy or finite critical Sobolev $\dot H^{s_c} \times \dot H^{s_c-1}$ norm, where $s_c=3/2-2/N$. However, our nonradial condition implies they are bounded in the Besov space $\dot B^{s_c}_{2, \infty} \times \dot B^{s_c-1}_{2, \infty}$, so the same is true for their linear evolution $u_{lin}$.

In the radial case, our condition implies that $\nabla u,~u_t \in L^{3N/(N+2), \infty}$ and, more generally, the first-order derivatives of $u_{lin}$ are dominated by those of the linear evolution of $Q_N$. However, below we shall focus more on the nonradial case.


These bounds do not prevent type II blow-up in the energy norm or imply scattering (and indeed $Q_N$ itself does not scatter) or energy conservation. Thus, in this generality these are only weak solutions. We shall have more to say about this in Corollary \ref{weak_solutions}. For more results about weak solutions to supercritical wave equations, see Ibrahim--Majdoub--Masmoudi \cite{imm}.


Conditions i and ii are scaling-invariant. Note that the second condition can be expressed in terms of the weighted $|x|^{-2-2/N} \dot W^{2, \infty} \times |x|^{-2-2/N} \dot W^{1, \infty}$ norm of the initial data, i.e.
\be\lb{cond_ii}
\||x|^{2+2/N}\Delta u_0\|_{L^\infty} + \||x|^{2+2/N}\nabla u_1\|_{L^\infty} \leq C_N^{N+1} = \bigg(\frac {2(N-2)}{N^2}\bigg)^{1 + \frac 1 N}.
\ee
The conclusion can also be expressed in terms of a weighted $|x|^{-2/N} L^\infty$ norm, i.e.~$\||x|^{2/N} u\|_{L^\infty_{t, x}} \leq C_N = \big(2(N-2)/N^2\big)^{1/N}$.

More generally, if $|\Delta u_0| + |\nabla u_1| \leq \alpha(-\Delta Q_N)$ with $\alpha \leq 1$, then $|u| \leq \alpha Q_n$, so the size of the solution depends linearly on the norm of the initial data.

Note that $C_N<1$ and $C_N \to 1$ as $N \to \infty$.


\begin{observation}
We actually only need a supersolution to equation (\ref{w}):
\be\lb{ineq}
-\Delta u_0 \geq u_0^{N+1},~u_0 \geq 0.
\ee
Then the solution $u$ with initial data $(u_0, 0)$ exists globally on $\R^{3+1}$, with $0 \leq u(x, t) \leq u_0(x)$, and so does any solution $v$ dominated by $u$, i.e.~with initial data $(v_0, v_1)$ such that $|\Delta v_0|+|\nabla v_1| \leq -\Delta u_0$.
%

Exactly the same criterion applies in the radial case, i.e.~if
$$
-\Delta u_0 = -(u_0)_{rr}-\frac 2 r (u_0)_r \geq u_0^{N+1}
$$
or equivalently $-(r u_0)_{rr} \geq r u_0^{N+1}$, then the solution $u$ with radial initial data $(u_0, 0)$ exists globally, with $0 \leq u(r, t) \leq u_0(r)$, and so does any radial solution $v$ dominated by $u$, i.e.~if $|(v_0)_r + v_0/r| + |v_1| \leq (u_0)_r + u_0/r$.

Much is known about the inequality (\ref{ineq}). For example, it has no positive solutions for $N \leq 2$, see \cite{bere}.

More relevantly, $Q_{N, \alpha} := C_N (r+\alpha)^{-2/N}$ is a supersolution for any $\alpha \geq 0$; check:
$$\begin{aligned}
-(r Q_{N, \alpha})_{rr} &= C_N(\frac 2 N(1-\frac 2 N)(\alpha+r)^{-1-2/N} + \alpha \frac 2 N)(1+\frac 2 N)(\alpha+r)^{-2-2/N}) \\
&\geq C_N^{N+1} r(\alpha+r)^{-2-2/N} = r Q_{N, \alpha}^{N+1}.
\end{aligned}
$$
Note that $Q_{N, \alpha}(x) = \inf_{|y|=\alpha} Q_N(x-y)$ is bounded for $\alpha>0$. Other supersolutions may also play a part in the study of this equation. This will be investigated in future papers.
\end{observation}

Even though the previous result did not deal with the asymptotic behavior of solutions, under slightly stronger conditions we can prove that solutions do not blow up in finite time and in several cases actually scatter. Some results only hold in the energy-supercritical range $N>4$. Also, for some estimates it is convenient to assume that $N$ is even.

The key part is having a uniform bound on the solutions. Thus, these results also apply to the solutions described in Corollary \ref{critic}. 

\begin{theorem}[Main result]\lb{thm_supercritic} Consider $\alpha>0$, $N>2$, and Borel measurable initial data $(u_0, u_1)$ to the equation (\ref{weakpm})
$$
u(t) = \cos(t\sqrt{-\Delta}) u_0 + \frac {\sin(t\sqrt{-\Delta})}{\sqrt{-\Delta}} u_1 \pm \int_0^t \frac {\sin((t-s)\sqrt{-\Delta})}{\sqrt{-\Delta}} |u(s)|^N u(s) \dd s.
$$
Assuming that
\be\lb{conditie}
|\Delta u_0(x)|+|\nabla u_1(x)| \leq Q_{N, \alpha}^{N+1} = \frac {C_N^{N+1}}{(\alpha+|x|)^{2+2/N}} = \bigg(\frac {2(N-2)}{N^2}\bigg)^{1 + \frac 1 N} \frac 1 {(\alpha+|x|)^{2+2/N}},
\ee
then the corresponding solution $u$ exists globally on $\R^{3+1}$ and $\|u\|_{L^\infty_{t, x}} \leq C_N \alpha^{-2/N}$; in fact $|u(x, t)| \leq Q_{N, \alpha} = C_N (\alpha+|x|)^{-2/N}$.

Moreover, $\|(u(t), u_t(t))\|_{\dot B^{s_c}_{2, \infty} \times \dot B^{s_c-1}_{2, \infty}} \les_\alpha \langle t \rangle^2$. More generally, if they are finite, Sobolev norms of $u$ grow at most polynomially. Finally, if $(u_0, u_1)$ have finite energy, then energy remains constant for all time.

Assuming that e.g.~$(u_0, u_1) \in \dot H^{s_c} \times \dot H^{s_c-1}$ (where $s_c=3/2-2/N$ is the critical Sobolev exponent), then the appropriate Strichartz norms of $u$ do not blow up in finite time.

In the defocusing case
, assuming that $N > 4$ and $(u_0, u_1) \in (\dot H^{s_c} \cap \dot H^1) \times (\dot H^{s_c-1} \cap L^2)$, or in the focusing case
, assuming in addition that $(u_0, u_1)$ are radially symmetric, then $u$ disperses ($\|u\|_{L^{2N}_{t, x}} < \infty$) and scatters.
\end{theorem}
One can estimate all the norms explicitly, except in the radially symmetric focusing case.

In the radial case, similar conclusions hold under the weaker hypothesis
$$
|(u_0)_r + u_0/r| + |u_1| \leq (Q_{N, \alpha})_r + Q_{N, \alpha}/r = (1-2/N) C_N (\alpha+|x|)^{-1-2/N}.
$$
We omit their statement and proof.


For the energy-critical equation (\ref{classical_pm}) with $N=4$, the defocusing case is already well understood: all finite energy solutions disperse and scatter. In the focusing case, assuming the initial data are radially symmetric, \cite{dkm2} proved the Soliton Resolution Conjecture, a trichotomy between blow-up, solitons, and scattering. Then, in this case, the solutions we constructed disperse and scatter: since the singular soliton $Q_4$ does not dominate the finite energy solitons $Q$, the latter cannot appear in the decomposition, which also precludes type II blow-up.

Condition (\ref{conditie}) is not scaling-invariant, but rescaling a solution only changes the value of $\alpha>0$. On the other hand, condition (\ref{conditie}) with $\alpha>0$ is optimal, since the singular soliton $Q_N$, which fulfills this condition with $\alpha=0$, provides a counterexample: its $L^{2N}_{t, x}$ norm is infinite on any time interval.

The $\dot H^{s_c} \times \dot H^{s_c-1}$ norm of the initial data can be arbitrarily high, though the $\dot B^{s_c}_{2, \infty}$ Besov norm is bounded by a constant of size $1$. The solutions constructed by Theorem \ref{thm_supercritic} are the first examples of large global solutions to a supercritical equation which correspond to generic initial data (from an open set), are bounded in the critical Sobolev norm, and scatter --- i.e.~the first ``normal" large solutions.

Finally, we return to the case $\alpha=0$, where approximating by smooth solutions yields some weaker conclusions.
\begin{corollary}\lb{weak_solutions} Assume that condition (\ref{conditie}) holds with $\alpha=0$ for $(u_0, u_1)$, that the equation is defocusing, and that $E[u](0) < \infty$. Then $E[u](t) \leq E[u](0)$ for all $t$. If the initial data only have locally finite energy, then so does the solution.
\end{corollary}
\begin{proof}[Proof of Corollary \ref{weak_solutions}]
Dropping the assumption that $\alpha>0$, take any admissible pair of initial data $(u_0, u_1)$ for $\alpha = 0$ and let $\Delta u_{n0} \to \Delta u_0$ and $\nabla u_{n1} \to \nabla u_1$ pointwise, where each $(u_{n0}, u_{n1})$ are smooth and satisfy condition (\ref{conditie}) for some $\alpha_n > 0$.
	
Due to dominated convergence (also see Proposition \ref{global_existence}), it follows that $(u_{n0}, u_{n1}) \to (u_0, u_1)$ in $\dot B^{s_c}_{2, \infty} \times \dot B^{s_c-1}_{2, \infty}$ and $u_n \to u$ in $L^{p_c, \infty}_x$ pointwise in $t$. By weak convergence, the latter implies that $E[u](t) \leq \limsup_n E[u_n](t)$ in the defocusing case.
	
However, all these approximations have constant energy, converging to $E[u](0)$, so $E[u](t)$ will be bounded.
\end{proof}



The paper is organized as follows: in the introduction we state the main results, in Section \ref{section_notations} we define some notations, in Section \ref{section_positivity} we state and prove the positivity criteria we subsequently use, and in Sections \ref{section_quadratic} and \ref{section_focusing} we prove the results pertaining to equation (\ref{eq_fritz}), respectively (\ref{eq_sup}).

\section{Notations}\lb{section_notations}

$A \les B$ means that $|A| \leq C |B|$ for some constant $C$. We denote various constants, not always the same, by $C$.

The Laplacian is the operator on $\R^3$ $\Delta=\frac {\partial^2}{\partial_{x_1}^2} + \frac {\partial^2}{\partial_{x_2}^2} + \frac {\partial^2}{\partial_{x_3}^2}$.

We denote by $L^p$ the Lebesgue spaces, by $\dot H^s$ homogenous and by $H^s$ inhomogenous Sobolev spaces, and by $L^{p, q}$ Lorentz spaces.

$\dot H^s$ and $H^s$ are Hilbert spaces and so is $\dot H^1 \times L^2$, under the norm
$$
\|(u_0, u_1)\|_{\dot H^1 \times L^2}=(\|u_0\|_{\dot H^1}^2+\|u_1\|_{L^2}^2)^{1/2}.
$$

For a radially symmetric function $u(x)$, we let $u(r):=u(x)$ for $|x|=r$.

We define the mixed-norm Strichartz spaces on $\R^3 \times [0, \infty)$
$$
L^p_t L^q_x := \Big\{f \mid \|f\|_{L^p_t L^q_x}:= \Big(\int_0^\infty \|f(x, t)\|_{L^q_x}^p \dd t\Big)^{1/p} < \infty \Big\},
$$
with the standard modification for $p=\infty$, and likewise for the reversed mixed-norm spaces $L^q_x L^p_t$. We use a similar definition for $L^p_t \dot W^{s, p}_x$. Also, for $I \subset [0, \infty)$, let $\|f\|_{L^p_t L^q_x(\R^3 \times I)} := \|\chi_I(t) f\|_{L^p_t L^q_x}$, where $\chi_I$ is the characteristic function of $I$.

The global Kato space is defined as follows:
$$
\mc K = \{u \in L^1_{loc} : \|u\|_{\mc K}:= \sup_y \int_{\R^3} \frac {|u(x)|}{|x-y|} < \infty\}.
$$

\section{Positivity criteria for solutions to the wave equation}\lb{section_positivity}

The positivity of the fundamental solution to the free wave equation in dimensions one to three has been known for a long time. In the context of semilinear wave equations, it was used by \cite{john} to prove finite time blow-up below the Strauss exponent for small initial data.

We use positivity in a different manner, by deriving sufficient --- and in some cases also necessary --- positivity criteria for solutions to the wave equation (inspired by the techniques used in \cite{becsof}). These are the basis of our main results.

The criteria can be stated either a priori, i.e. for classical solutions, or for more general rough solutions, which will be only a.e.~positive. Here we shall state them a priori, while also pointing out the necessary modifications for rough solutions.

Although not needed in this paper, we first look at the simplest case, that of the one-dimensional equation.

\begin{lemma}\lb{1d} Consider smooth $(u_0, u_1)$ and suppose that there exists an antiderivative $\partial_x^{-1} u_1$ of $u_1$ such that $u_0 \geq |\partial_x^{-1} u_1|$. Then the corresponding solution $u$ of the one-dimensional free wave equation
$$
u_{tt}-u_{xx}=0,~u(0)=u_0,~u_t(0)=u_1
$$
is nonnegative on $\R^{1+1}$.

Moreover, assume that $u_0$ and $u_1$ are just Borel measurable and $u_0 \geq |\partial_x^{-1} u_1|$ a.e.. Then $u$ is nonnegative, possibly except for a measure zero set of light rays.
\end{lemma}
Note that, in order to apply this to compact support solutions, it is necessary that $\int_\R u_1 \dd x = 0$.

When $u_0$ and $\partial_x^{-1} u_1$ have compact support, the condition that $u_0 \geq |\partial_x^{-1} u_1|$ is also necessary for positivity.

\begin{proof}[Proof of Lemma \ref{1d}] By the well-known d'Alembert formula, $u(x, t) = u_+(r-t)+u_+(r+t)$, where
$$
u_\pm(x) = \frac 1 2 (u_0(x) \mp \partial_x^{-1} u_1(x))
$$
and $\partial_x^{-1} u_1$ denotes any antiderivative of $u_1$. If $u_\pm$ are nonnegative, then so is $u$. Hence it suffices that
$u_0(x) \mp \partial_x^{-1} u_1(x) \geq 0$, which is our hypothesis.
\end{proof}

Skipping the two-dimensional case for now, we next state a positivity criterion for radially symmetric solutions to the free wave equation on $\R^{3+1}$.

The proof is based on the following construction introduced in \cite{becsof}: for radial $u(r)=u(|x|)$, define
\be\lb{tu}
T(u)(r) = (ru(r))',~u(r)= \frac 1 r \int_0^r T(u)(s) \dd s.
\ee
Then, $u$ solves the free wave equation on $\R^{3+1}$ if and only if $U=T(u)$ solves the free wave equation on $[0, \infty) \times \R$
\be\lb{1D}
U_{tt}-U_{rr}=0,~U(0)=U_0=T(u_0),~U_t(0)=U_1=T(u_1),
\ee
with Neumann boundary conditions $U_r(0, t) = 0$. Equation (\ref{1D}) has solutions of the form (for $r \geq 0$)
$$
U(r, t)=\chi_{r \geq t} U_+(r-t) + \chi_{r \leq t} U_-(t-r) + \chi_{r+t \geq 0} U_-(r+t) + \chi_{r+t \leq 0} U_+(-r-t),
$$
where by d'Alembert's formula
\be\lb{dalembert}
U_\pm(r)=\frac 1 2(U_0(r)\mp\partial_r^{-1} U_1(r)).
\ee

If we take $(U(t_0), U_t(t_0))$ as initial data in equation (\ref{1D}), we obtain a time-translated solution $\tilde U$ with
$$
\tilde U_\pm(r) = \frac 1 2 (U(r, t_0) \mp \partial_r^{-1} U_t(r, t_0)).
$$
This decomposition is related to the original one by
\be\lb{dalembert'}
\tilde U_-(r)=\chi_{r \geq 0}U_-(r+t_0),~\tilde U_+(r)=\chi_{r \geq t_0} U_+(r-t_0) + \chi_{0 \leq r \leq t_0} U_-(t_0-r)
\ee
for $t_0 \geq 0$ and similar for $t_0 \leq 0$.

In \cite{becsof}, we also used this construction to define incoming and outgoing radial initial data.
\begin{definition}\lb{def_outgoing} We call a radially symmetric pair $(u_0, u_1)$ \emph{outgoing} if
$$
(u_0)_r + \frac {u_0} r = u_1.
$$
\end{definition}
In this definition, the only way for both $u_0$ and $u_1$ to be smooth is if $u_0(0) = 0$, but this condition is not necessary for rougher, discontinuous solutions.

Given outgoing initial data, we showed that for $r \geq t \geq 0$
\be\lb{positive_outgoing}
u(r, t) = \frac {r-t} r u_0(r-t)
\ee
and $u(r, t) \equiv 0$ for $0 \leq r \leq t$. It is easy to see that, for outgoing initial data, if $u_0 \geq 0$ then $u \geq 0$. More generally, if $u_0 \geq 0$ a.e., then $u \geq 0$ everywhere, possibly except for a measure zero set of forward light cones.

The three-dimensional positivity criterion involves one more derivative than the one-dimensional one.

\begin{lemma}\lb{positive} Suppose that $(u_0, u_1)$ are smooth and radial. Then $(r u_0(r))_r > r |u_1|$ if and only if the solution $u$ to the free wave equation
$$
u_{tt}-\Delta u =0,~u(0)=u_0,~u_t(0)=u_1
$$
is positive on $\R^{3+1}$. Moreover, in this case for any $t \in \R$ $(r u(r, t))_r > r |u_t(t, r)|$.

Conversely, if there exists $r_0 \geq 0$ such that $(r u_0)_r(r_0) \leq r_0 |u_1(r_0)|$, then $u \leq 0$ somewhere in $\R^{3+1}$ (more precisely, $u(0, t=r_0) \leq 0$ or $u(0, t=-r_0) \leq 0$).

In addition, $u \in L^\infty_{t, x}$ if and only if $(r u_0(r))_r,~r u_1 \in L^\infty$ and in this case
$$
\|(r u(r, t))_r\|_{L^\infty_{t, r}} + \|r u_t(t, r)\|_{L^\infty_{t, r}} \leq 2 \|u\|_{L^\infty_{t, x}} \leq 2(\|(r u_0(r))_r\|_{L^\infty} + \|r u_1\|_{L^\infty}).
$$
\end{lemma}
Similar conclusions hold everywhere, possibly except at $r=0$, if we just assume $u_0$, $u_1$, and $(u_0)_r$ are Borel measurable instead of smooth and the inequalities only hold a.e.~at $t=0$.

Our positivity condition implies that $(ru_0(r))' > 0$. Note that $u_0$ will decay no faster than $1/r$ and (of course) will be positive. In particular, this means that $u_0$ can have finite energy, but cannot have finite $L^2$ norm.

\begin{observation} If we allow equality as well, then $(r u_0(r))_r \geq r |u_1|$ is equivalent to $u \geq 0$ on $\R^{3+1}$ (and to $(r u(r, t))_r \geq r |u_t(t, r)|$ for any $t \in \R$).
\end{observation}

\begin{proof}[Proof of Lemma \ref{positive}] With $T$ defined by (\ref{tu}), let $T(u)=U$. Then, by virtue of (\ref{tu}), in order to prove that $u > 0$ it suffices to prove that $U > 0$ and in turn this follows once we show that $U_\pm > 0$. However, by (\ref{dalembert}),
\be\lb{vpm}
U_\pm(r) = \frac 1 2 (U_0(r) \mp \partial_r^{-1} U_1(r)) = \frac 1 2 \big((ru_0(r))_r \mp ru_1(r)\big) > 0.
\ee
By formula (\ref{dalembert'}), this implies that $\tilde U_\pm > 0$, so, by a computation analogous to (\ref{vpm}), we get that $(r u(r, t))_r \mp ru_t(r, t) >0$.

For the converse statement, if $(r u_0)_r(r_0) \leq r_0|u_1(r_0)|$, then either $U_+(r_0) \leq 0$ or $U_-(r_0) \leq 0$. Both cases imply that $U(0, t) \leq 0$ for some $t$ ($t=-r_0$ in the first case, $t=r_0$ in the second case). But by (\ref{tu}) $u(0, t)=U(0, t)$, so $u$ is indeed nonpositive somewhere in $\R^{3+1}$.

The same reasoning applies to the $L^\infty_{t, x}$ norm: following (\ref{tu}) and (\ref{vpm}),
$$
\|u\|_{L^\infty_{t, x}} \leq \|U\|_{L^\infty_{t, r}} \leq \|U_-\|_{L^\infty} + \|U_+\|_{L^\infty} \leq \|(r u_0(r))_r\|_{L^\infty} + \|r u_1\|_{L^\infty}.
$$

Conversely, again following (\ref{vpm}),
$$
\|(r u_0(r))_r\|_{L^\infty} + \|r u_1\|_{L^\infty} \leq 2\max(\|U_-\|_{L^\infty}, \|U_+\|_{L^\infty}) \leq 2 \|U(0, t)\|_{L^\infty_t} \leq 2 \|u\|_{L^\infty_{t, x}}.
$$
(Again we used the fact that $u(0, t)=U(0, t)$). Same is true for any other time $t \in \R$, so
$$
\|(r u(r, t))_r\|_{L^\infty_{t, r}} + \|r u_t(t, r)\|_{L^\infty_{t, r}} \leq 2 \|u\|_{L^\infty_{t, x}}.
$$
\end{proof}

In the proof we also use the following more refined condition for the boundedness of solutions.
\begin{corollary}\lb{boundedness} Suppose $u$ is a smooth radial solution to the free wave equation on $\R^{3+1}$:
$$
u_{tt} - \Delta u=0,~u(0)=u_0,~u_t(0)=u_1.
$$
Then $a \leq u \leq b$ on $\R^{3+1}$ if and only if
$$
(r u_0(r))_r - a \geq r|u_1|,~b - (r u_0(r))_r \geq r|u_1|,
$$
in which case it is also true that for any $t \in \R$
$$
(r u(r, t))_r - a \geq |u_t(r, t)|,~b - (r u(r, t))_r \geq |u_t(r, t)|.
$$
\end{corollary}

These two conditions taken together imply that $a \leq (r u(r, t))_r \leq b$ and that $|u_t(r, t)| \leq \frac {b-a} {2r}$.

Corollary \ref{boundedness} is natural, in the sense that its hypotheses are true for all radially symmetric Schwartz class $(u_0, u_1)$ and some finite $a$, $b \in \R$. However, it is not necessary to assume any decay at infinity.

\begin{proof}[Proof of Corollary \ref{boundedness}] Apply Lemma \ref{positive} to $u-a$ and to $b-u$.
\end{proof}

We next state two positivity and boundedness criteria that hold for nonradial solutions. These criteria are less sharp than Lemma \ref{positive}, in the sense that they are sufficient, but not necessary. For one criterion we need to assume that the initial data $(u_0, u_1)$ decay at infinity together with their derivatives. In addition, this criterion requires two derivatives instead of one.

In the statement of the boundedness criterion we also use the global Kato space, defined as follows:
$$
\mc K = \{u : \|u\|_{\mc K}:= \sup_y \int_{\R^3} \frac {|u(x)|}{|x-y|} < \infty\}.
$$

Due to the pairing between $L^{3/2, 1}$ and $L^{3, \infty}$ and to the fact that $\frac 1 {|x|} \in L^{3,\infty}$, it follows that $L^{3/2, 1} \subset \mc K$.

Also note that, due to the boundedness of the Riesz transforms, $\Delta u \in L^{3/2, 1}$ is equivalent to $D^2 u \in L^{3/2, 1}$.

\begin{lemma}\lb{positive_nonradial} Consider a smooth solution $u$ to the free wave equation on $\R^{3+1}$ with initial data $(u_0, u_1)$. If $u_0 > 0$ and $u_1 \geq |\nabla u_0|$, then $u > 0$ on $\R^3 \times [0, \infty)$.

Alternatively, assume that $(u_0, u_1)$ decay at infinity together with their derivatives. If $-\Delta u_0 > |\nabla u_1|$, then $u > 0$ on $\R^{3+1}$. Also, if $\Delta u_0, \nabla u_1 \in \mc K$, then $u \in L^\infty_{t, x}$ and
\be\lb{Linftynonradial}
\|u\|_{L^\infty_{t, x}} \leq \frac 1 {4\pi} (\|\Delta u_0\|_{\mc K} + \|\nabla u_1\|_{\mc K}).
\ee
\end{lemma}

The condition $-\Delta u_0 > |\nabla u_1| \geq 0$ automatically implies that $u_0 > 0$ (in fact, $u$ is superharmonic), since we can write
\be\lb{elliptic}
u_0(x) = \frac 1 {4\pi} \int_{\R^3} \frac{-\Delta u_0(y)}{|x-y|} \dd y.
\ee
For the same reason, if $-\Delta u_0 \in \mc K$, then $u_0 \in L^\infty$.

More generally, it suffices e.g.~that the first condition holds a.e.~on each two-dimensional sphere or that the second condition holds a.e..

It is easy to generate pairs of initial data that satisfy our hypotheses by, for example, first choosing $u_1$, then choosing $-\Delta u_0 > |\nabla u_1|$, and then retrieving $u_0$ by means of the formula (\ref{elliptic}).

If we assume instead that $u_0 \geq 0$ or that $-\Delta u_0 \geq |\nabla u_1|$ (i.e.~we allow for equality), then we get that $u \geq 0$.

A sharper way of bounding $u$ is to say that it is dominated by the positive solution with initial data $((-\Delta)^{-1}(|\Delta u_0|+|\nabla u_1|), 0)$. Also note that a solution $u$ with initial data $(u_0, 0)$ and $-\Delta u_0 \geq 0$ has $u_t \leq 0$ for $t \geq 0$ and $u_t \geq 0$ for $t \leq 0$ (so it reaches its maximum at time $0$).

By the method of descent (i.e.~writing a solution of the two-dimensional wave equation as a solution of the three-dimensional wave equation that is constant in one variable) the same criteria also hold in the two-dimensional case.

A sharper positivity criterion can be written in terms of the Radon transform of the initial data, but it is nonlocal and harder to use in the proof.

\begin{proof}[Proof of Lemma \ref{positive_nonradial}] This follows immediately from the usual (D'Alembert) solution formula for the free wave equation in three dimensions. With no loss of generality take $x=0$; then for $t > 0$
\be\lb{fund_solution}
u(0, t) = \frac 1 {4\pi t^2} \int_{|y|=t} u_0(y) \dd y + \frac 1 {4\pi t} \int_{|y|=t} (u_0)_r(y) + u_1(y) \dd y.
\ee
Note that $|(u_0)_r| \leq |\nabla u_0|$.

A similar formula ensures the validity of the second criterion: for $t \geq 0$
\be\lb{delta1}
u(0, t) = \frac 1 {4\pi} \int_{|y| \geq t} \frac {-\Delta u_0(y)}{|y|} \dd y + \frac 1 {4\pi t} \int_{|y|=t} u_1(y) \dd y
\ee
and
\be\lb{delta2}
\frac 1 {4\pi t} \int_{|y| = t} u_1(y) \dd y = - \frac t {4\pi} \int_t^\infty \int_{S^2} (u_1)_r(r\omega) \dd \omega \dd r = - \frac 1 {4\pi} \int_{|y| \geq t} \frac {t(u_1)_r(y)} {|y|^2} \dd y.
\ee
Then note that $|(u_1)_r| \leq |\nabla u_1|$ and $\frac t {|y|^2} \leq \frac 1 {|y|}$.

The boundedness estimate (\ref{Linftynonradial}) follows for example from formulas (\ref{delta1}) and (\ref{delta2}). Also note that an even more general statement was already proved in \cite{becgol}. 
\end{proof}

It is useful to have a criterion for the pointwise convergence of sequences of free wave equation solutions in this class. We distinguish two regimes, namely monotone convergence and dominated convergence.
\begin{lemma}\lb{pointwise} Consider a sequence of initial data $(u_{n0}, u_{n1})$ such that
$$
0 \leq -\Delta u_{n0} - |\nabla u_{n0}| \leq c(x) \in L^1_{loc}
$$
and
$$
-\Delta u_{n+1\,0} + \Delta u_{n0} \geq |\nabla u_{n+1\,1} - \nabla u_{n1}|
$$
for each $n$. Then $(u_{n0})$ converges monotonically and $(u_{n1})$ converges to some Borel measurable functions $(u_0, u_1)$ with $c(x) \geq -\Delta u_0 \geq \nabla u_1$ and the linear evolution $(u_n)_{lin}$ converges monotonically to $u_{lin}$.
	
Next, suppose that $u$ is a fixed positive a.~e.~finite solution with initial data $(u_0, u_1)$, $-\Delta u_0 - |\nabla u_0| \geq 0$. Let $v_n$ be a sequence of solutions dominated by $u$, i.e.
$$
-\Delta u_0 - |\nabla u_0| \geq |\Delta v_{n0}| + |\nabla v_{n1}|.
$$
If $\Delta v_{n0} \to \Delta v_0$, $\nabla v_{n1} \to \nabla v_1$ pointwise, then $v_n \to v$ pointwise, where $v$ is the solution with initial data $(v_0, v_1)$.
\end{lemma}
The condition that $c(x) \in L^1_{loc}$ is a technical one and not truly necessary. We omit the statements of other pointwise convergence conditions, for the other classes of positive solutions. The proof is based on a direct application of Lebesgue's monotone and dominated convergence theorems.

Finally, we also state a boundedness criterion that only holds locally in time, but requires fewer conditions.
\begin{lemma}\lb{local_bound} Consider the free wave equation on $\R^{3+1}$ with smooth initial data $u_0$ and $u_1$, such that $u_0,~\nabla u_0,~u_1 \in L^\infty$. Then the corresponding solution $u$ satisfies the bounds
$$\begin{aligned}
\sup_{\R^3 \times [-T, T]} u \leq \sup u_0 + T (\|u_1\|_{L^\infty} + \|\nabla u_0\|_{L^\infty}) \\
\inf_{\R^3 \times [-T, T]} u \geq \inf u_0 - T (\|u_1\|_{L^\infty} + \|\nabla u_0\|_{L^\infty}).
\end{aligned}$$
\end{lemma}
\begin{proof}[Proof of Lemma \ref{local_bound}] This immediately follows from (\ref{fund_solution}).
\end{proof}

\section{Large solutions to equation \ref{eq_fritz}}\lb{section_quadratic}

\begin{proof}[Proof of Proposition \ref{fritz}] Consider a smooth solution $u$ of equation (\ref{eq_fritz}). By setting $v=e^{-u}$, we obtain the free wave equation on $\R^{3+1}$
$$
v_{tt}-\Delta v = 0,~v(0)=v_0=e^{-u_0},~v_t(0)=v_1=-e^{-u_0} u_1
$$
for $v$, where the initial data are also smooth.

The free wave equation always has a smooth solution $v$ for smooth initial data. Then, the transformation can be reversed as long as $v>0$, by setting $u=\ln v$. By Lemma \ref{positive}, $v$ is guaranteed to be positive as long as $(r v_0(r))_r > r|v_1|$. Expressing this in terms of $u$, we obtain exactly condition (\ref{cond}).

By the same lemma, in this case $(r v(r, t))_r > r |v_t(r, t)|$ or in other words $u_r(r, t) + |u_t(r, t)| < \frac 1 r$.

Next, $u \in L^\infty_{t, x}$ is equivalent to $v \in L^\infty_{t, x}$ and $\inf v > 0$. Applying Lemma \ref{boundedness}, we see that it is necessary and sufficient that
$$
\sup (r v_0(r))_r + r|v_1| < \infty,~\inf (r v_0(r))_r - r |v_1| > 0.
$$
But
$$
(r v_0(r))_r + r|v_1| = e^{-u_0} (1 -r (u_0)_r + r |u_1|) \leq e^{-\inf u_0} (1 + \|r (u_0)_r\|_{L^\infty} + \|r u_1\|_{L^\infty})
$$
and
$$
(r v_0(r))_r - r |v_1| = e^{-u_0} r \Big(\frac 1 r - (u_0)_r - |u_1|\Big) \geq e^{-\sup u_0} \epsilon.
$$
Thus $v$ is always between these two bounds: $A\leq v \leq B$, where $A=e^{-\sup u_0} \epsilon$ and $B=e^{-\inf u_0} (1 + \|r (u_0)_r\|_{L^\infty} + \|r u_1\|_{L^\infty})$. Taking the logarithm, we obtain exactly the inequality (\ref{comparison}).

In this situation, Lemma \ref{boundedness} also implies that
$$
A \leq (r v(r, t))_r - r |v_t(r, t)| \leq (r v(r, t))_r + r |v_t(r, t)| \leq B,
$$
hence
$$
2 r |v_t(r, t)| \leq B-A.
$$
Taking into account the fact that $A \leq v \leq B$, these statements imply
$$
u_r(r, t) + |u_t(r, t)| \leq \frac 1 r \Big(1-\frac A B\Big),~-u_r(r, t) + |u_t(r, t)| \leq \frac 1 r \Big(\frac B A -1\Big),
$$
respectively
$$
|u_t(r, t)| \leq \frac 1 {2r} \Big(\frac B A  - 1\Big).
$$
Thus we retrieve the conclusions (\ref{detaliat}).
\end{proof}

\begin{proof}[Proof of Proposition \ref{prop_dispersion}]
Concerning energy and dispersion, assuming that $u_0 \in L^\infty$ and that $(u_0)_r + |u_1| \leq \frac {1-\epsilon} r$, then we have already proved above that $v \geq e^{-\sup u_0} \epsilon$, i.e.~$u \leq \sup u_0 + \ln (1/\epsilon)$. Since $(u_0, u_1) \in \dot H^1 \times L^2$,
$$
\|(v_0-1, v_1)\|_{\dot H^1 \times L^2} \leq e^{-\inf u_0} \|(u_0, u_1)\|_{\dot H^1 \times L^2}.
$$
We need to subtract $1$ because $\lim_{x \to \infty} v_0 = e^0 = 1$. As a radial solution of the free wave equation, $v-1$ conserves energy,
$$
\|(v(t)-1, v_t(t)\|_{\dot H^1 \times L^2} = \|(v_0-1, v_1)\|_{\dot H^1 \times L^2},
$$
and also satisfies the endpoint Strichartz estimate \cite{klma}:
$$
\|v-1\|_{L^2_t L^\infty_x} \les \|(v_0 - 1, v_1)\|_{\dot H^1 \times L^2}.
$$
In order to convert back to $u$, note that
$$
\|(u(t), u_t(t))\|_{\dot H^1 \times L^2} \leq e^{\sup u} \|(v(t)-1, v_t(t))\|_{\dot H^1 \times L^2}
$$
and $|u| \leq |v-1| \max(1, e^{\sup u})$. Putting together all these estimates we obtain exactly (\ref{est}).
\end{proof}

\begin{proof}[Proof of Proposition \ref{blowup}] Let $v=e^{-u}$. Then, as stated above, $v$ satisfies the free wave equation
$$
v_{tt}-\Delta v = 0,~v(0)=v_0=e^{-u_0},~v_t(0)=v_1=-e^{-u_0} u_1.
$$
Our hypothesis $(u_0)_r(r_0)+|u_1(r_0)| \geq \frac 1 {r_0}$ implies that $(r v_0)_r(r_0) \leq r_0 |v_1(r_0)|$. Then, by Lemma \ref{positive}, it follows that $v(0, t=r_0) \leq 0$ or $v(0, t=-r_0) \leq 0$.

Suppose $v(0, t=r_0) \leq 0$ and let $t_0 = \inf \{t \geq 0 : \exists r,~|r| \leq r_0-t \text{ and } v(r, t)=0\}$. By continuity and compactness, $v(r_1, t_0)=0$ for some $r_1$. Then clearly $t_0 \leq r_0$ and $t_0>0$ (since at time $0$ $v_0=e^{u_0} \ne 0$). By our definition, on the light cone with $\{(r, t) : t \geq 0,~|r-r_1| \leq t_0 - t \}$ $v$ is positive.

This means that we can take $u=\ln v$ and retrieve a smooth solution $u$ of the original equation (\ref{eq_fritz}) on this cone. At the same time, since $\lim_{(r, t) \to (r_1, t_0)} v(r, t) = v(r_1, t_0) = 0$, it follows that $\lim_{(r, t) \to (r_1, t_0)} u(r, t) = -\infty$, which implies the $L^\infty_{loc}$ blow-up.

Consider $x_0 \in \R^3$ with $|x_0|=r_1$, so that $v(x_0, t_0)=0$. It follows that, for $|x-x_0|, |t-t_0| \leq 1$, $|v(x, t)| \les |x-x_0| + |t-t_0|$.

Therefore, under the same conditions and on the domain of $u$,
\be\lb{blowup_bound}
u(x, t) \leq C + \ln(|x-x_0| + |t-t_0|).
\ee
Setting $x=x_0$, we get that $\|u(t)\|_{L^\infty_x(|x-x_0| \leq 1)} \geq C + |\ln |t-t_0||$.

Concerning the $H^{3/2}$ norm, we use the Trudinger-Moser inequality, see (\cite{trudinger}) and (\cite{moser}): for any sufficiently regular bounded domain $\Omega \subset \R^3$ there exists $C_\Omega$ such that
$$
\int_\Omega \exp\bigg(\bigg(\frac {|u(x)|}{C_\Omega\|u\|_{H^{3/2}}}\bigg)^2\bigg) - 1 \dd x \leq 1.
$$
By making the coordinate change $x-x_0=|t-t_0|(y-x_0)$, we obtain that $\|u(t)\|_{H^{3/2}(|x-x_0| \leq 1)} \geq C |\ln|t-t_0||^{1/2}$.
\end{proof}

\begin{proof}[Proof of Proposition \ref{prop_fritz_general}]
Here we make the substitution $v=F(u)$, where $F''/F'=-f$, so we can take $F(u)=\int_0^u e^{-\int_0^s f(\sigma) \dd \sigma} \dd s$, see (\ref{F}). Then $v$ solves the free wave equation
$$
v_{tt}-\Delta v = 0,~v(0)=v_0=F(u_0),~v_t(0)=v_1=F'(u_0)u_1.
$$

If $F(\pm \infty)=\pm \infty$ we can always invert this transformation, so we obtain global smooth solutions for any smooth initial data. On the other hand, if $F(-\infty)=a \in \R$ and/or $F(+\infty)=b \in \R$, then, in order to invert by taking $u=F^{-1}(v)$, we need to impose the condition $v > a$ and/or $v<b$. By Corollary \ref{boundedness}, in the first case it suffices that
$$
r F'(u_0) (u_0)_r + F(u_0)-a > r F'(u_0) |u_1|,
$$
which is equivalent to condition (\ref{cond1}).

In order to obtain $L^\infty_{t, x}$ solutions, we must ask that $v \in L^\infty_{t, x}$ and, depending on each case, $\inf v > a$ and/or $\sup v < b$.

Suppose that $u_0 \in L^\infty$. By Lemma \ref{positive}, in order for $v \in L^\infty_{t, x}$ we must assume that $r (u_0)_r,~r u_1 \in L^\infty$.

By Corollary \ref{boundedness}, a necessary and sufficient condition for e.g.~$\inf v - a > 0$ is that
$$
\inf r F'(u_0) (u_0)_r + F(u_0) - a - r F'(u_0) |u_1| > 0.
$$
Rewriting this condition as
$$
F'(u_0)(r (u_0)_r - r|u_1|) \geq \epsilon_0 + a - F(u_0)
$$
and noting that
$$
F'(u_0)=e^{-\int_0^{u_0} f(\sigma) \dd \sigma},\ F(u_0)-a=\int_{-\infty}^{u_0} e^{-\int_0^s f(\sigma) \dd \sigma} \dd s,
$$
we obtain
$$
-(u_0)_r + |u_1| \leq \frac 1 r \bigg( \int_{-\infty}^{u_0} e^{\int_s^{u_0} f(\sigma) \dd \sigma} \dd s - \frac {\epsilon_0}{F'(u_0)} \bigg).
$$
Taking into account the fact that we are assuming $u_0 \in L^\infty$, so $F'(u_0)$ is bounded from above and below, this reduces to the stated condition (\ref{Linfty}).

Finally, $\inf v-a>0$ equally implies (\ref{Linfty'}) uniformly for all times $t$, but with a different $\epsilon$ depending on the whole of $v$, not just the initial data.
\end{proof}

Finally, the nonradial case is not so different from the radial case.

\begin{proof}[Proof of Proposition \ref{prop_nonradial}] The proof is almost identical to that of Proposition \ref{prop_fritz_general}. We use the same substitution $v=F(u)$, with $F$ given by (\ref{F}). Then $v$ must be a solution of the free wave equation
$$
v_{tt}-\Delta v = 0,~v(0)=v_0=F(u_0),~v_t(0)=v_1=F'(u_0) u_1,
$$
which is guaranteed to exist.

This transformation is always invertible when $F(\pm \infty)=\pm\infty$, but in the other cases we need to check whether $v>F(-\infty)=a$ and/or $v<F(+\infty)=b$. This is done using the positivity criteria of Lemma \ref{positive_nonradial}.

If $F(-\infty)=a \in \R$, then our conditions will in fact imply that $v \geq \inf v_0$ or that $v \geq 0$. Stated in terms of $v$, the conditions are $\inf v_0 > -\infty$ and $v_1 \geq |\nabla v_0|$ or $(v_0, v_1)$ Schwartz functions and $-\Delta v_0 \geq |\nabla v_1|$. Note that
$$\begin{aligned}
\nabla v_0 &= F'(u_0) \nabla u_0,\\
\Delta v_0 &= F'(u_0) \Delta u_0 + F''(u_0) (\nabla u_0)^2,\\
\nabla v_1 &= F'(u_0) \nabla u_1 + F''(u_0) u_1 \nabla u_0,
\end{aligned}
$$
and that by definition $F''/F'=-f$ and $F'>0$. We obtain exactly condition (\ref{cond_nonradial1}).

If $D^2 u_0 \in L^{3/2, 1} \subset \mc K$, then $u_0 \in L^\infty$, so $F'(u_0)$ and $F''(u_0) \in L^\infty$. Assuming that $D^2 u_0,~\nabla u_1 \in L^{3/2, 1}$, it also follows that $\nabla u_0, u_1 \in L^{3, 1}$. Consequently $\Delta v_0,~\nabla v_1 \in L^{3/2, 1}$, so by Lemma \ref{positive_nonradial} $v$ is bounded. Under our previous conditions, this also implies that $u$ is bounded.
\end{proof}

\begin{proof}[Proof of Proposition \ref{prop_general}] Let $v=F(u)$, where $F$ is given by (\ref{F}). Note that by our definition $F(0)=0$. We obtain the free wave equation on $\R^{3+1}$ for $v$:
$$
v_{tt}-\Delta v=0,~v(0)=v_0=F(u_0),~v_t(0)=v_1=F'(u_0)u_1.
$$

It is easy to see that, since $u_0$ and $u_1$ are of Schwartz class, then so are $v_0-F(0)=v_0$ and $v_1$. Therefore the solution $v$ is globally defined on $\R^{3+1}$, $v(t)$ is of Schwartz class for each $t \in \R$, and $v$ disperses.

In particular, by the usual decay estimates, $\lim_{t \to \infty} \|v(x, t)\|_{L^\infty_x} = 0$. Since $v$ is continuous and for each fixed $t$ $\lim_{x \to \infty} v(x, t) = 0$, it follows that $\lim_{(x, t) \to \infty} v(x, t) = 0$.

Without loss of generality, assume $F(-\infty)=a \in \R$ and $F(+\infty)=b \in \R$; then $a<F(0)=0$ and $b>0$. From the above it follows that there exists some $R>0$ such that if $|(x, t)| > R$ then $v(x, t) \in (a/2, b/2)$. Since the set $\{(x, t): |(x, t)| \leq R\}$ is compact, $v$ reaches its maximum and minimum on this set, i.e.~$m=\min_{|(x, t)| \leq R} v \leq v \leq M = \max_{|(x, t)| \leq R} v$.

There are two cases. If $m \geq a$ or $M \leq b$, then the solution $u$ blows up in finite time; see the proof of Proposition \ref{blowup} for more details. Otherwise, one has $a<\inf v < \sup v<b$.

In the latter case, one can invert the transformation by taking $u=F^{-1}(v)$ and obtain a global smooth solution $u$ on $\R^{3+1}$ to equation (\ref{eq_fritz_general}). In addition, $u=F^{-1}(v)$ is bounded and more generally $(F^{-1})^{(n)}$ is bounded on the domain of $v$.

We obtain that $|u| \les |v|$ and $|D^n u| \les \sum_{k_1+\ldots+k_n=n} |D^{k_1} v| \ldots |D^{k_n} v|$ (with constants that may depend on the solution). Therefore $u(t)$ is a Schwartz function for each $t$, its Sobolev $H^n$ norms are uniformly bounded, and $u$ and all its derivatives disperse, because $v$ has these properties and dominates $u$.

Concerning scattering, note that $u=(F^{-1})'(0) v + O(v^2)$. The first term is a solution of the free wave equation, while the second term goes to zero in any $H^n$ Sobolev norm, since $v(t)$ is uniformly bounded in $H^n$ and $\lim_{t \to \infty} \|D^n v(t)\|_{L^\infty} = 0$ for any $n \geq 0$.
\end{proof}

\begin{proof}[Proof of Proposition \ref{prop_local_existence}] The proof uses the same ideas as that of Proposition \ref{prop_nonradial}. We make the transformation $v=F(u)$, where $F$ is given by (\ref{F}). Then $v$ is a solution of the free wave equation on $\R^{3+1}$
$$
v_{tt}-\Delta v = 0,~v(0)=v_0=F(u_0),~v_t(0)=v_1=F'(u_0)u_1.
$$
Since $u_0$ and $u_1$ are smooth functions, so are $v_0$ and $v_1$, leading to a smooth solution $v$ on $\R^{3+1}$.

In order to reverse the transformation, we need to ensure that $v > a$. Our hypotheses and Lemma \ref{local_bound} guarantee this on $\R^3 \times (-T, T)$, where $T$ is defined by (\ref{time}). Also, for any $t < T$, we see that $v \in L^\infty_{t, x}(\R^3 \times [-t, t])$ and $\inf v > a$ on $\R^3 \times [-t, t]$, which implies that $u \in L^\infty_{t, x}(\R^3 \times [-t, t])$ as well.
\end{proof}

\section{Large solutions to equation (\ref{eq_sup})}\lb{section_focusing}
%
%

\begin{proof}[Proof of Proposition \ref{maximum}] Consider the following sequence:
$$\begin{aligned}
&u^0(t)=\cos(t\sqrt{-\Delta})u_0+\frac{\sin(t\sqrt{-\Delta})}{\sqrt{-\Delta}}u_1, \\
&u^{n+1}(t)=\cos(t\sqrt{-\Delta})u_0 + \frac {\sin(t\sqrt{-\Delta})} {\sqrt{-\Delta}}u_1 + \int_0^t \frac{\sin((t-s)\sqrt{-\Delta})} {\sqrt{-\Delta}} |u^n(s)|^N u^n(s) \dd s.
\end{aligned}$$
The first term in the sequence is nonnegative due to our hypothesis, by (\ref{positive_outgoing}) --- possibly except for a measure zero set of forward light cones, in this case ---, Lemma \ref{positive}, or Lemma \ref{positive_nonradial}. Inductively we see that all $u^n$ are smooth and nonnegative, hence all the integrals are well-defined.

Furthermore, one proves by induction that the sequence $(u^n)_{n \geq 0}$ is monotonically increasing, due to the positivity of the kernel
$$
\frac {\sin(t\sqrt{-\Delta})}{\sqrt{-\Delta}}(x, y) = \frac 1 {4\pi t} \delta_{|x-y|=t}.
$$

Since the sequence $(u^n)_n$ is monotonically increasing, it must have a limit in $[0, +\infty]$ pointwise on $\R^{3+1}$. Let $u:=\lim_{n\to\infty}u^n$; clearly,
$$
\lim_{n \to \infty} |u^n|^N u^n = |u|^N u,
$$
with the usual convention that $(+\infty)^{N+1}=+\infty$.

By the monotone convergence theorem it follows that
$$
u(t)=\cos(t\sqrt{-\Delta})u_0 + \frac {\sin(t\sqrt{-\Delta})} {\sqrt{-\Delta}}u_1 + \int_0^t \frac{\sin((t-s)\sqrt{-\Delta})} {\sqrt{-\Delta}} |u(s)|^N u(s) \dd s,
$$
i.e. $u$ is a solution to (\ref{eq_sup}) with initial data $(u_0, u_1)$.

In order to compare two solutions, note that if $u_0 \geq v_0 \geq 0$ in case i, $(ru_0(r))_r \pm ru_1 \geq (rv_0(r))_r \pm rv_1 \geq 0$ in case ii, etc., then $u^0 \geq v^0 \geq 0$ and by induction $u^n \geq v^n \geq 0$ for every $n$, so $u \geq v \geq 0$.
\end{proof}

\begin{proof}[Proof of Proposition \ref{global_existence}] It is easy to prove by induction that $|v^n - v^{n-1}| \leq u^n - u^{n-1}$ and $|v^n| \leq u^n$ for each $n$. But the sequence $u^n(x, t)$ is increasing, so it converges whenever $u(x, t)$ is finite to the solution $\tilde u \leq u$ of (\ref{weak}) produced by Proposition \ref{maximum}.
	
Hence, whenever $u(x, t)$ is finite, $v^n(x, t)$ is a Cauchy sequence and converges to some limit $v(x, t)$. This is true regardless of whether the nonlinearity is focusing or defocusing. Passing to the limit on both sides, by dominated convergence it follows that $v$ is a solution of (\ref{weak}).

Assume that $(v_m)_{lin} \to v_{lin}$ pointwise a.e.~on each backward light cone. By induction and dominated convergence, then $v_n^m \to v^m$ for each $m$. On the other hand, $v_m^n(x, t) \to v_m(x, t)$ uniformly in $m$, due to their domination by $u$: $|v_m^n - v_m^{n-1}| \leq u^n - u^{n-1}$. Hence $v_m \to v$ pointwise a.e.~on each backward light cone.

Finally, suppose that the equation is defocusing, but the solution $v$ we constructed is not necessarily unique; consider some other solution $\tilde v$. Then both $u-v=V$ and $u-\tilde v = \tilde V$ will be nonnegative mild solutions of the equation
$$
V_{tt} - \Delta V = N(u) - N(u-V),~V(0)=V_0=u_0-v_0,~V_t(0)=V_1=u_1-v_1,
$$
with $N(u) = |u|^N u$, i.e.
$$
V(t) = \cos(t\sqrt{-\Delta}) V_0 + \frac {\sin(t\sqrt{-\Delta})}{\sqrt{-\Delta}} V_1 + \int_0^t \frac {\sin((t-s)\sqrt{-\Delta})}{\sqrt{-\Delta}} (N(u(s))-N(u(s)-V(s))) \dd s.
$$
This is a nonlinear equation to which we can apply the proof of Proposition \ref{maximum}, so it has a smallest nonnegative solution $V_{min}$ such that $V,\, \tilde V \geq V_{min}$.

However, both $v$ and $V_{min}$ are obtained by iteration, as the limits  of some sequences $v^n$ and $V^n$. It is easy to prove by induction that in fact $V^n = u - v^n$ for each $n$, so in the limit $V_{min} = u - v = V$. Therefore $\tilde V \geq V$, so $\tilde v \leq v$.
	
Likewise, $W=v+u$ and $\tilde W = \tilde v + u$ are nonnegative mild solutions of the equation 
$$
W_{tt} - \Delta W = N(W-u) + N(u),
$$
which has a smallest positive solution $W_{min}$, so $W,\, \tilde W \geq W_{min}$. Again, in fact $W_{min} = v+u$, so $\tilde v \geq v$.
	
Since $\tilde v \leq v$ and $\tilde v \geq v$, it follows that $\tilde v  = v$.
	
Obviously, this uniqueness argument is inspired by the usual proof of the dominated convergence theorem using Fatou's lemma.
\end{proof}

\begin{proof}[Proof of Corollary \ref{critic}] We use comparison with the solutions with initial data $((1\pm\epsilon)Q, 0)$.

A straightforward computation shows that
$$
\frac {d}{d\epsilon} E[(1\pm \epsilon) Q] \mid_{\epsilon=0} = 0,~\frac {d^2}{d\epsilon^2} E[(1\pm \epsilon) Q] \mid_{\epsilon=0} < 0,
$$
so $E[(1\pm \epsilon) Q] < E[Q]$ for small $\epsilon>0$. Using the criterion of \cite{kenig}, it follows that the solution with initial data $((1-\epsilon) Q, 0)$ disperses, while the solution with initial data $((1+\epsilon)Q, 0)$ blows up in finite time.

As an aside, note that our criteria guarantee that any solution with initial data $(CQ, 0)$, $C>0$, is positive.

It is well-known (see \cite{kenig}) that the finiteness of the $L^8_{t, x}$ norm implies that the solution must exist globally on $\R^{3+1}$. 
\end{proof}

\begin{proof}[Proof of Corollary \ref{supercritic}]
All conclusions save the last one follow directly from Proposition \ref{global_existence}.
	
Since the nonlinearity $|u|^N u$ is in $L^\infty_t L^{\frac{3N}{2(N+1)}, \infty}_x \subset L^\infty_t \dot B^{s_c-2}_{2, \infty\, x}$, for $N>2$ it also is in $L^1_{loc} \cap H^{-3/2}_{loc}$ and for $N>4$ it is in $L^{6/5}_{loc} \subset H^{-1}_{loc}$. Therefore its contribution is in $L^1_{loc} \cap H^{-1/2}_{loc}$ and for $N>4$ also in $L^2_{loc}$.
	
On the scale of Besov spaces, the nonlinearity's contribution is in $\dot B^{s_c-1}_{2, \infty\, x}$, but with a norm that can grow linearly with time:
$$
\|u(t)-u_{lin}(t)\|_{\dot B^{s_c-1}_{2, \infty\, x}} \les t.
$$
\end{proof}

\begin{proof}[Proof of Theorem \ref{thm_supercritic}] Condition (\ref{conditie}) implies that the solution $u$ is dominated, in the sense of Corollary \ref{global_existence}, by $Q_N(x-x_0)$ for any $x_0$ with $|x_0|=\alpha$. It immediately follows that $|u(x, t)| \leq C_N (\alpha+|x|)^{-2/N}$.
	
In the radial case, we compare $u$ directly with the supersolution $Q_{N, \alpha}$.

Next, we show that the solutions we constructed are well approximated by strong and even classical solutions. We use the fact that $\dot H^s \cap L^\infty$ is a Banach algebra for e.g.~$0 \leq s \leq 3/2$, hence so is $\dot B^s_{2, \infty}$ for $0<s<3/2$.
	
Thus, since $\|u_{lin}\|,\, \|u\|_{L^\infty_{t, x}} \les \alpha^{-2/N}$ and
$$
\|u(t)-u_{lin}(t)\|_{\dot B^{s_c-1}_{2, \infty}} \les t,~\|u(t)\|_{\dot B^{s_c}_{2, \infty} + \dot B^{s_c-1}_{2, \infty}} \les \langle t \rangle,
$$
it follows that, at least for even $N$,
$$
\||u(t)|^N u(t)\|_{\dot B^{s_c}_{2, \infty} + \dot B^{s_c-1}_{2, \infty}} \les \alpha^{-1} \langle t \rangle.
$$
In particular, the nonlinearity is in $L^2_{loc}$ when $N>4$. Plugging this bound back into the equation, we obtain that
$$
\|u-u_{lin}\|_{\dot B^{s_c+1}_{2, \infty} + \dot B^{s_c}_{2, \infty}} \les \alpha^{-1} t \langle t \rangle.
$$
Combining the two bounds, we get that
$$
\|u(t)-u_{lin}(t)\|_{\dot B^{s_c}_{2, \infty}} \les t (1+ \alpha^{-1} \langle t \rangle),~\|u(t)\|_{\dot B^{s_c}_{2, \infty}} \les 1 + t (1+ \alpha^{-1} \langle t \rangle),
$$
and moreover
$$
\|u_t(t)-u_{lin\, t}(t)\|_{\dot B^{s_c-1}_{2, \infty}} \les t (1+ \alpha^{-1} \langle t \rangle),~\|u_t(t)\|_{\dot B^{s_c-1}_{2, \infty}} \les t (1+ \alpha^{-1} \langle t \rangle).
$$
We can bootstrap this as many times as needed. In fact, since $u$ is uniformly bounded, the only limiting factor is the regularity of the initial data. If $(u_0, u_1) \in \dot H^s \times \dot H^{s-1}$, $s \geq 0$, it is easy to show along the same lines that e.g.~$\|(u(t), u_t(t))\|_{\dot H^s \times \dot H^{s-1}}$ is polynomially bounded, for each fixed $\alpha>0$. Thus, our solutions preserve regularity.
	
Consequently, when $\alpha>0$, smooth initial data give rise to classical solutions. In particular, for classical solutions local energy remains finite, if it is so to begin with, is monotone along light cones, and total energy is constant --- again, if it is finite to begin with.
	
In the defocusing case, energy conservation implies that both the free energy
$$
E_0[u](t) = \int_{\R^3} \frac {|u_t|^2 + |\nabla u|^2} 2 \dd x
$$
and the nonlinear energy $\|u\|_{L^{N+2}}^{N+2}$ are uniformly bounded. In the focusing case, we can only tell that they are growing at most polynomially, at a rate that depends on $\alpha>0$.
	
However, even in the focusing case, if $s_c<1$ then $L^{p_c, \infty} \cap L^\infty \subset L^{N+2}$, so the nonlinear term (hence the free energy also) is bounded with a bound that depends only on $\alpha$:
$$
\|u\|_{L^{N+2}}^{N+2} \les \alpha^{1-4/N},~\|(u, u_t)\|_{\dot H^1 \times L^2} \les \|(u, u_t)\|_{\dot H^1 \times L^2} + \alpha^{1-4/N}.
$$
	
When $s_c>1$, again taking into account that $u \in L^\infty$, it follows that $u$ always has locally finite energy:
$$
e(x, t) :=  \frac {|u_t|^2 + |\nabla u|^2} 2 \pm \frac {|u|^{N+2}}{N+2} \in L^1_{loc\, x}.
$$
This is not guaranteed when $s_c \leq 1$, unless we assume that $(u_0, u_1) \in H^1_{loc} \times L^2_{loc}$.
	
We now return to rough solutions and show that they are well approximated by classical solutions. Note that for each solution $\|u\|_{L^{3N}} \les \alpha^{-1/N}$. By Gronwall's inequality, it follows that for any two of our solutions $u$ and $\tilde u$ and for $s \geq 0$
$$
\|(u, u_t) - (\tilde u, \tilde u_t)\|_{\dot H^s \times \dot H^{s-1}} \les e^{|t|/\alpha} \|(u_0, u_1) - (\tilde u_0, \tilde u_1)\|_{\dot H^s \times \dot H^{s-1}}.
$$
	
If $N>4$ and $s_c>1$, then $L^{p_c, \infty} \subset L^{N+2}_{loc}$, so our rough solutions have locally finite energy. If $N \leq 4$ and $s_c \leq 1$, then $H^1_{loc} \subset L^{N+2}_{loc}$ and $\dot H^1 \cap \dot B^{s_c}_{2, \infty} \subset L^{N+2}$. Assuming that $(u_0, u_1) \in H^1_{loc} \times L^2_{loc}$, solutions $u$ again have locally finite energy and if $(u_0, u_1) \in H^1 \times L^2$ then all the energy terms grow at most polynomially.
	
In both cases, approximation by smooth solutions and energy monotonicity along light cones can be used to prove that energy is conserved and local energy does not increase.

In order to prove that the solution $u$ does not blow up, it is easiest to assume that the initial data have compact support. Due to the finite speed of propagation, $u$ will have compact support at any fixed time $t \in \R$. Since it is bounded, all its Lebesgue norms will be bounded, uniformly on compact time intervals. Therefore the Strichartz norms of $u$ will also be finite on compact time intervals.

In the defocusing case, if the initial data have finite energy, then the Morawetz inequality
$$
\int_\R u(x_0, t)^2 \dd t + \int_{\R^{3+1}} \frac {|\not \!\nabla u|^2} {|x-x_0|} + \int_{\R^{3+1}} \frac {u^{N+1}}{|x-x_0|} \les E[u]
$$
(where $\not \! \nabla = \nabla - \frac {x-x_0}{|x-x_0|}(\frac {x-x_0}{|x-x_0|} \cdot \nabla)$) guarantees that $\|u\|_{L^\infty_x L^2_t}^2 \les E[u] < \infty$. On the other hand, we already know that $\|u\|_{L^\infty_{t, x}} \les \alpha^{-2/N}$ and
$$
\|u\|_{L^{3N/2, \infty}_x L^\infty_t} \leq \|Q_N\|_{L^{3N/2, \infty}} \les 1.
$$
Interpolating between the three bounds, we first obtain that $\|u\|_{L^p_x L^\infty_t} < \infty$ for $3N/2<p \leq \infty$, then that $\|u\|_{L^{2N}_{t, x}}< \infty$ (assuming that $N>4$), so $u$ disperses and scatters.

A more general argument to prove that the solution does not blow up in finite time is the following: for $0 \leq s_c \leq 3/2$, $\dot H^{s_c} \cap L^\infty$ is a Banach algebra and
$$
\|f g\|_{\dot H^{s_c}} \les \|f\|_{\dot H^{s_c}} \|g\|_{L^\infty} + \|f\|_{L^\infty} \|g\|_{\dot H^{s_c}}.
$$
Then
$$
\||u(t)|^Nu(t)\|_{\dot H^{s_c}} \les \|u(t)\|_{\dot H^{s_c}} \|u\|_{L^\infty_{t, x}}^N.
$$
Furthermore,
$$
\||u(t)|^Nu(t)\|_{\dot H^{s_c-1}} \les \|u(t)\|_{\dot H^{s_c}} \|u\|_{L^\infty_{t, x}}^{N/2} \|u\|_{L^\infty_t L^{p_c, \infty}_x}^{N/2}.
$$
Since
$$
\bigg\|\frac {\sin(t\sqrt{-\Delta})}{\sqrt{-\Delta}} f\bigg\|_{\dot H^{s_c}} \leq |t| \|f\|_{\dot H^{s_c}},
$$
it follows that for $t \geq 0$
$$
\|(u(t), u_t(t)\|_{\dot H^{s_c} \times \dot H^{s_c-1}} \les \|(u_0, u_1)\|_{\dot H^{s_c} \times \dot H^{s_c-1}} + \int_0^t (t-s) \|(u(s), u_t(s)\|_{\dot H^{s_c} \times \dot H^{s_c-1}} \|u\|_{L^\infty_{t, x}}^N \dd s.
$$
By Gronwall's inequality $\|(u(t), u_t(t)\|_{\dot H^{s_c}} \les \|(u_0, u_1)\|_{\dot H^{s_c} \times \dot H^{s_c-1}} \exp(Ct^2\|u\|_{L^\infty_{t, x}}^N)$. Using this inequality for time $t \leq 1$ and iterating, we get that
$$
\|(u(t), u_t(t)\|_{\dot H^{s_c}} \les \|(u_0, u_1)\|_{\dot H^{s_c} \times \dot H^{s_c-1}} \exp(Ct\|u\|_{L^\infty_{t, x}}^N).
$$

Here $\|u\|_{L^\infty_{t, x}}^N$ can be replaced by $\|u\|_{L^\infty_{t, x}}^{N/2} \|u\|_{L^\infty_t L^{p_c, \infty}_x}^{N/2} \les \alpha$. In fact, our previous arguments show that Sobolev norms can grow at most polynomially.

Thus, if $\|(u_0, u_1)\|_{\dot H^{s_c} \times \dot H^{s_c-1}} < \infty$, then $\|u(t)\|_{\dot H^{s_c}}$ is bounded on compact time intervals. This immediately implies that the nonlinearity $\||u(t)|^N u(t)\|_{\dot H^{s_c}}$ is also bounded on compact time intervals and the same for Strichartz norms.

Following the result of \cite{duro}, for radial solutions in either the focusing or the defocusing case, the boundedness of the critical Sobolev norm on compact time intervals implies that the solution disperses and scatters.
\end{proof}

\section*{Acknowledgments}
We would like to thank Wilhelm Schlag for the useful discussions and the anonymous referee for some suggestions.

This work was partially supported by a grant from the Simons Foundation (\#429698, Marius Beceanu) and by startup funds provided by the University at Albany (SUNY). Part of this work was done while M.B.\;was visiting the University of Chicago.

This work was partially supported by a grant from the
Simons Foundation (\#395767 to Avraham Soffer). A.S.\;is partially supported by NSF grant DMS 1201394. Part of this work was done while A.S.\;was visiting at CCNU, Wuhan, China.


\begin{thebibliography}{ABC11}
\bibitem[Bec]{bec} M.~Beceanu, \emph{A center-stable manifold for the energy-critical wave equation in $\R^3$ in the symmetric setting}, J.~Hyper.~Differential Equations, Vol.~11, Issue 3, 437 (2014).

\bibitem[BeGo]{becgol} M.~Beceanu, M.~Goldberg, \emph{Strichartz estimates and maximal operators for the wave equation in $\R^3$}, Journal of Functional Analysis, Vol.~266, Issue 3, 1 February 2014, pp.~1476--1510.

\bibitem[BeSo1]{becsof} M.~Beceanu, A.~Soffer, \emph{Large outgoing solutions to supercritical wave equations}, preprint, arXiv:1601.06335.

\bibitem[BeSo2]{becsof2} M.~Beceanu, A.~Soffer, \emph{Large initial data global well-posedness for a supercritical wave equation}, preprint, arXiv:1602.08163.

\bibitem[BCDN]{bere} H. Berestycki, I.~Capuzzo Dolcetta, L.~Niremberg, \emph{Superlinear indefinite elliptic problems
and nonlinear Liouville theorems}, Topoi.~Methods Nonlinear Anal.~4 (1993), pp.~59--78. 

\bibitem[BeL\"o]{bergh} J.~Bergh, J.~L\"ofstr\"om, \emph{Interpolation Spaces. An Introduction}, Springer-Verlag, 1976.

\bibitem[Bul1]{bul1} A.~ Bulut, \emph{The radial defocusing energy-supercritical cubic nonlinear wave equation in $\R^{1+5}$}, preprint, arXiv:1104.2002.

\bibitem[Bul2]{bul2} A.~Bulut, \emph{Global well-posedness and scattering for the defocusing energy-supercritical cubic nonlinear wave equation}, preprint, arXiv:1006.4168.

\bibitem[Bul3]{bul3} A.~Bulut, \emph{The defocusing energy-supercritical cubic nonlinear wave equation in dimension five}, preprint, arXiv:1112.0629.

\bibitem[Chr]{christo} D.~Christodoulou, \emph{Global solutions of nonlinear hyperbolic equations for small initial data}, Comm.~Pure Appl.~Math.~1986, 39, pp.~267--282.

\bibitem[DoLa]{dola} B.~Dodson, A.~Lawrie, \emph{Scattering for radial, semi-linear, super-critical wave equations with bounded critical norm}, Archive for Rational Mechanics and Analysis, December 2015, Vol.~218, Issue 3, pp.~1459--1529.

\bibitem[DKM1]{dkm} T.~Duyckaerts, C.~Kenig, F.~Merle, \emph{Classification of the radial solutions of the focusing, energy-critical wave equation}, Cambridge Journal of Mathematics, Vol.~1 (2013), No.~1, pp.~75--144.

\bibitem[DKM2]{dkm2} T.~Duyckaerts, C.~Kenig, F.~Merle, \emph{Scattering for radial, bounded solutions of focusing supercritical wave equations}, preprint, arXiv:1208.2158.

\bibitem[DuRo]{duro} T.~Duyckaerts, T.~Roy, \emph{Blow-up of the critical Sobolev norm for nonscattering radial solutions of supercritical wave equations on $\R^3$}, Bull.~Soc.~Math.~France
145 (3), 2017, pp.~503--573.




\bibitem[GSV]{gsv} J.~Ginibre, A.~Soffer, G.~Velo, \emph{The global Cauchy problem for the critical nonlinear wave equation}, J.~Funct.~Anal.~110 (1992), pp.~96--130.

\bibitem[Gri]{grillakis} M.~Grillakis, \emph{Regularity and asymptotic behaviour of the wave equation with a
critical nonlinearity}, Ann.~of Math.~132 (1990), pp.~485--509.

\bibitem[IMM]{imm} S.~Ibrahim, M.~Majdoub, N.~Masmoudi, \emph{Well- and ill-posedness issues for energy supercritical waves}, Analysis and PDE, Vol.~4, No.~2, 2011.

\bibitem[Joh]{john} F.~John, \emph{Blow-up of solutions of non-linear wave equations in three dimensions}, Manuscript.\
Math.~28 (1979), pp.~235-–268.

\bibitem[Kel]{kel} J.~Keller, \emph{On solutions of nonlinear wave equations}, Comm.~Pure Appl.~Math.~10 (1957), pp.~523-530.

\bibitem[KeMe1]{kenig} C.~Kenig, F.~Merle, \emph{Global well-posedness, scattering and blow-up for the energy-critical focusing non-linear wave equation}, Acta Mathematica (2008), Vol.~201, Issue 2, pp.~147--212.

\bibitem[KeMe2]{keme2} C.~Kenig, F.~Merle, \emph{Nondispersive radial solutions to energy supercritical non-linear wave equations, with applications}, American Journal of Mathematics (2011), Vol.~133, No.~4, pp.~1029--1065.

\bibitem[KiVi1]{kivi1} R.~Kilip, M.~Visan, \emph{The radial defocusing energy-supercritical nonlinear wave equation in all space dimensions}, Proceedings of the American Mathematical Society
Vol.~139, No.~5 (2011), pp.~1805--1817.

\bibitem[KiVi2]{kivi2} R.~Kilip, M.~Visan, \emph{The defocusing energy-supercritical nonlinear wave equation in three space dimensions}, Transactions of the American Mathematical Society
Vol.~363, No.~7 (2011), pp.~3893--3934.

\bibitem[Kla1]{klainerman} S.~Klainerman, \emph{Global existence for nonlinear wave equations}, Commun.~Pure Appl.~Math., 33 (1980), pp.~43--101.

\bibitem[Kla2]{klainer} S.~Klainerman, \emph{The Null Condition and Global Existence to Nonlinear Wave Equations}, Lect.~Appl.~Math.~1986, 23, pp.~293--326.

\bibitem[KlMa]{klma} S.~Klainerman, M.~Machedon, \emph{Space-time estimates for null forms and the local existence theorem}, Comm.~Pure Appl.~Math, 46 (1993), pp.~1221--1268.

\bibitem[KNS1]{kns} J.~Krieger, K.~Nakanishi, W.~Schlag, \emph{Global dynamics away from the ground state for the energy-critical nonlinear wave equation}, American Journal of Mathematics, Vol.~135, No.~4, August 2013, pp.~935--965.

\bibitem[KNS2]{kns3} J.~Krieger,~K.~Nakanishi, W.~Schlag, \emph{Global dynamics of the nonradial energy-critical wave equation above the ground state energy}, Discrete and Continuous Dynamical Systems - Series A, Vol.~33, Issue 6, June 2013, pp.~2423--2450.

\bibitem[KNS3]{kns2} J.~Krieger, K.~Nakanishi, W.~Schlag, \emph{Center-stable manifold of the ground state in the energy space for the critical wave equation}, Mathematische Annalen,
February 2015, Vol.~361, Issue 1, pp.~1--50.

\bibitem[KrSc1]{krsc} J.~Krieger, W.~Schlag, \emph{On the focusing critical semi-linear wave equation}, American Journal of Mathematics, Vol.~129, No.~3 (Jun., 2007), pp.~843--913.

\bibitem[KrSc2]{krsc2} J.~Krieger, W.~Schlag, \emph{Large global solutions for energy supercritical nonlinear wave equations on $\R^{3+1}$}, Journal d'Analyse Math\'ematique Vol.~133, pp.~91--131 (2017).

\bibitem[KST]{kst} J.~Krieger, W.~Schlag, D.~Tataru, \emph{Slow blow-up solutions for the $H^1(\R^3)$ critical focusing semilinear wave equation}, Duke Math.~J.~Vol.~147, No.~1 (2009), pp.~1--53. 

\bibitem[Lev]{levine} H.~Levine, \emph{Instability and nonexistence of global solutions to nonlinear wave
equations of the form $Pu_{tt} = Au + \mc F(u)$}, Trans.~Amer.~Math.~Soc.~192 (1974), pp.~1--21.

\bibitem[Li]{li} D.~Li, \emph{Global wellposedness of hedgehog solutions for the $(3+1)$ Skyrme model}, preprint, arXiv:1208.4977.

\bibitem[Lin]{lindblad2} H.~Lindblad, \emph{A sharp counterexample to the local existence of low-regularity solutions to nonlinear wave equations}, Duke Math.~J.~Vol.~72, No.~2 (1993), pp.~503--539.

\bibitem[LiSo]{lindblad} H.~Lindblad, C.~Sogge, \emph{On existence and scattering with minimal regularity for semilinear wave equations}, J.~Funct.~Anal.~130 (1995), pp.~357-–426.

\bibitem[LOY]{loy} J.~Luk, S.-J.~Oh, S.~Yang, \emph{Solutions to the Einstein-scalar-field system in spherical symmetry with large bounded variation norms}, preprint, arXiv:1605.03893.

\bibitem[MPY]{mpy} S.~Miao, L.~Pei, P.~Yu, \emph{On classical global solutions of nonlinear wave equations with large data}, International Mathematics Research Notices, Vol.~2019, Issue 19, 2019, pp.~5859--5913.

\bibitem[MiZh]{miao} C.~Miao, J. Zheng, \emph{Scattering theory for energy-supercritical Klein-Gordon equation}, DCDS-s, 2016, Vol.~9, Issue 6, pp.~2073--2094.

\bibitem[Mos]{moser} J.~Moser, \emph{A Sharp form of an inequality by N. Trudinger}, Indiana Univ.~Math.~20 (1971), pp.~1077--1092.

\bibitem[Pec]{pecher} H.~Pecher, \emph{Nonlinear small data scattering for the wave and Klein-Gordon equation}, Math.~Z.~185 (1984), pp.~261--270.

\bibitem[Roy1]{roy1} T.~Roy, \emph{Scattering above energy norm of solutions of a loglog energy-supercritical Schr\"{o}dinger equation with radial data}, preprint, arXiv:0911.0127.

\bibitem[Roy2]{roy2} T.~Roy, \emph{Global existence of smooth solutions of a 3D loglog energy-supercritical wave equation}, Anal.~PDE 2(3), pp.~261--280 (2009).

\bibitem[ShSt]{shst} J.~Shatah, M.~Struwe, \emph{Well-posedness in the energy space for semilinear wave equations with critical growth}, Internat.~Math.~Res.~Notices 7 (1994), pp.~303--309.

\bibitem[Str1]{struwe} M.~Struwe, \emph{Globally regular solutions to the $u^5$ Klein-Gordon equation}, Ann.~Scuola Norm.~Sup.~Pisa Cl.~Sci.~15 (1988), pp.~495--513.

\bibitem[Str2]{struwe2} M.~Struwe, \emph{Global well-posedness of the Cauchy problem for a super-critical nonlinear wave equation in two space dimensions}, Mathematische Annalen 350.3 (2011), pp.~707--719.

\bibitem[Tao1]{tao} T.~Tao, \emph{Global regularity for a logarithmically supercritical defocusing nonlinear wave equation for spherically symmetric data}, J.~Hyperbolic Diff.~Eq., 4, 2007, pp.~259--266.

\bibitem[Tao2]{tao2} T.~Tao, \emph{Finite time blowup for a supercritical defocusing nonlinear wave system}, Anal.~PDE 9(8), pp.~1999--2030 (2016).

\bibitem[Tru]{trudinger} N.~S.~Trudinger, \emph{On imbeddings into Orlicz spaces and some applications}, J.~Math.~Mech.~17 (1967), pp.~473--483.

\bibitem[WaYu]{wang} J.~Wang, P.~Yu, \emph{A large data regime for nonlinear wave equations}, Journal of the European Mathematical Society, Vol.~18, Issue 3, 2016, pp.~575--622.

\bibitem[Yan]{shiwu} S.~Yang, \emph{Global solutions of nonlinear wave equations with large data}, Selecta Mathematica (2015) 21:4, pp.~1405--1427.

\end{thebibliography}
\end{document}